\documentclass[11pt, twoside,leqno]{article}

\usepackage{amsmath}
\usepackage{amssymb}
\usepackage{titletoc}
\usepackage{mathrsfs}
\usepackage{amsthm}

\usepackage{indentfirst}
\usepackage{color}
\usepackage{txfonts}

\allowdisplaybreaks

\def\bint{{\ifinner\rlap{\bf\kern.30em--}
\int\else\rlap{\bf\kern.35em--}\int\fi}\ignorespaces}

\def\sbint{{\ifinner\rlap{\bf\kern.32em--}
\hspace{0.078cm}\int\else\rlap{\bf\kern.45em--}\int\fi}\ignorespaces}

\def\red{\color{red}}

\def\rr{\mathbb{R}}
\def\rn{\mathbb{R}^n}
\def\rp{\mathbb{R}^{n+1}_+}

\def\nn{\mathbb{N}}
\def\zz{\mathbb{Z}}
\def\zn{\mathbb{Z}^n}

\def\ls{\lesssim}

\def\fz{\infty}

\def\cf{{\mathcal F}}

\def\cs{{\mathcal S}}

\def\bh{\mathbb{H}}

\def\lp{{L^p(\rn)}}

\def\r{\right}
\def\lf{\left}

\def\r{\right}
\def\lf{\left}

\def\f{\frac}

\def\loc{{\mathop\mathrm{\,loc\,}}}

\DeclareMathOperator*{\esssup}{ess\,\sup}

\def\eqref#1{(\ref{#1})}

\newtheorem{theorem}{Theorem}[section]
\newtheorem{lemma}[theorem]{Lemma}
\newtheorem{corollary}[theorem]{Corollary}
\newtheorem{proposition}[theorem]{Proposition}

\theoremstyle{definition}
\newtheorem{remark}[theorem]{Remark}
\newtheorem{definition}[theorem]{Definition}
\newtheorem{assumption}[theorem]{Assumption}

\numberwithin{equation}{section}

\begin{document}

\title{\bf\Large Riesz Transform Characterization of
Hardy Spaces Associated with Ball Quasi-Banach Function Spaces
\footnotetext{\hspace{-0.35cm} 2020 {\it Mathematics Subject
Classification}. {Primary 42B30; Secondary 42B35, 42B20, 44A15, 47B06.}
\endgraf {\it Key words and phrases}. Riesz transform characterization, ball quasi-Banach
function space, Hardy space, Poisson integral.
\endgraf This project is partially supported by the
National Key Research and Development Program of China (Grant No. 2020YFA0712900)
and the National Natural Science Foundation of China
(Grant Nos. 11971058, 12071197 and 12122102).}}
\date{ }
\author{Fan Wang, Dachun Yang\,\footnote{Corresponding author, E-mail: \texttt{dcyang@bnu.edu.cn}/{\red August 22,
2022}/Final version.}\, \ and Wen Yuan}
\maketitle

\vspace{-0.8cm}

\begin{center}
\begin{minipage}{13cm}
{\small {\bf Abstract}\quad Let $X$ be a ball quasi-Banach
function space satisfying some mild assumptions
and $H_X(\mathbb{R}^n)$ the Hardy space associated
with $X$. In this article, the authors introduce both the
Hardy space $H_X(\mathbb{R}^{n+1}_+)$ of harmonic
functions
and the Hardy space $\mathbb{H}_X(\mathbb{R}^{n+1}_+)$
of harmonic vectors, associated with $X$,
and then establish the isomorphisms among
$H_X(\mathbb{R}^n)$, $H_{X,2}(\mathbb{R}^{n+1}_+)$,
and $\mathbb{H}_{X,2}(\mathbb{R}^{n+1}_+)$,
where $H_{X,2}(\mathbb{R}^{n+1}_+)$
and $\mathbb{H}_{X,2}(\mathbb{R}^{n+1}_+)$ are,
respectively, certain subspaces of  $H_X(\mathbb{R}^{n+1}_+)$
and $\mathbb{H}_X(\mathbb{R}^{n+1}_+)$.
Using these isomorphisms, the authors establish
the first order Riesz transform characterization of $H_X(\mathbb{R}^n)$.
The higher order Riesz transform characterization of $H_X(\mathbb{R}^n)$ is  also obtained.
The results obtained in this article
have a wide range of generality and can be applied to
the classical Hardy space, the weighted Hardy space,
the Herz--Hardy space, the Lorentz--Hardy space,
the variable Hardy space, the mixed-norm
Hardy space, the local
generalized Herz--Hardy space, and the mixed-norm Herz--Hardy space.
}
\end{minipage}
\end{center}

\vspace{0.3cm}

\section{Introduction}

Let $\cs(\rn)$ be the set of all \emph{Schwartz functions}
on $\rn$, equipped with the well-known topology
determined by a countable family of norms. Recall
that, for any $j\in\{1,\dots,n\}$, the \emph{$j$-th Riesz
transform $R_j$} is defined by setting, for any
$f\in \cs(\rn)$
and $x\in\rn$,
$$
R_j(f)(x):=\lim_{\delta\to0^+}c_{(n)}\int_{\{y\in\rn:\,|y|>\delta\}}\frac{y_j}{|y|^{n+1}}f(x-y)\,dy,
$$
here and thereafter, $\delta\to 0^+$ means that
$\delta\in(0,\infty)$ and $\delta\to 0$, and
\begin{equation}\label{unit-b}
c_{(n)}:=\frac{\Gamma([n+1]/2)}{\pi^{(n+1)/2}}
\end{equation}
with $\Gamma$ denoting the \emph{Gamma function}.
It is well known that Riesz transforms are natural
generalizations of the Hilbert transform to the Euclidean
space of higher dimension and the most typical
examples of Calder\'on--Zygmund operators,  and have
many interesting and useful properties (see, for instance,
\cite{gl,s70,s93} and their references).
Indeed, they are the simplest, nontrivial, ``invariant"
operators under the acting of the group of rotations
on the Euclidean space $\rn$, and they also constitute
typical and important examples of Fourier multipliers.
Moreover, they can be used to mediate between various
combinations of partial derivatives of functions. All these properties
make Riesz transforms ubiquitous in mathematics;
see \cite{s70} for more details on their applications.

Let $\mathscr{M}(\rn)$ be the set of all measurable functions on $\rn$.
Recall that the Lebesgue space $L^p(\rn)$ with $p\in(0,\fz)$ is defined by setting
$$L^p(\rn):=\left\{f \in \mathscr{M}(\rn):\
\|f\|_{\lp}:=\left[\int_{\rn}|f(x)|^p\,dx\right]^{1/p}<\infty\right\}.$$
A very classical result about Riesz transforms is that
they are bounded on Lebesgue space $L^p(\rn)$ with $p\in(1,\infty)$.
However, when $p\in(0,1]$,
the Riesz transform is not bounded on $L^p(\rn)$ anymore.
As a natural generalization and a natural substitute
of $L^p(\rn)$ with $p\in(0,1]$, the Hardy space $H^p(\rn)$
was originally initiated  by Stein and Weiss \cite{sw} and
then systematically investigated by Fefferman and Stein \cite{fs}.
The celebrated articles \cite{sw} and \cite{fs} inspire many
new ideas for the real-variable theory of function spaces.
For instance, various real-variable  characterizations
of classical Hardy spaces reveal the important connections
among various concepts in harmonic analysis, such as harmonic functions,
various maximal functions, and various square functions.

It is well known that Riesz transforms are not only
bounded on Hardy spaces, but also characterize
Hardy spaces.   As a famous feature of Hardy spaces,
the  characterization   via Riesz transforms was first studied by
Fefferman and Stein  \cite{fs} in 1972 and further extended
by Wheeden \cite{w76} to the weighted Hardy spaces. From then on,
there exists a series of studies on the Riesz transform characterization of various
Hardy type spaces. As a recent progress in this direction, we reall that
Cao et al. \cite{ccyy16} established the
Riesz transform characterization of Musielak--Orlicz Hardy
spaces, and Yang et al. \cite{yzn} established
the Riesz transform characterization of variable exponent
Hardy spaces. We also refer the reader to \cite{bbd21,ds22,dl21,dlwy21,ll22,mv22}
for some recent developments on the study of Riesz transforms
and to \cite{bd21,g22,j21,lyyy,ttd21,yy15} for some research of Riesz transforms associated with operators.

A key step in  establishing the
Riesz transform characterization of Hardy spaces $H^p(\rn)$ is to
extend the elements of $H^p(\rn)$ to the
upper half space
$
\rp:=\rn\times(0,\infty)
$
via the Poisson integral. This extension in turn has
a close relation with the analytical definition
of $H^p(\rn)$, which is the key starting point of studying
the Hardy space, before people paid attention to the
real-variable theory of $H^p(\rn)$ (see, for instance, \cite{s69,sw}).
Recall also that the real-variable theory of $H^p(\rn)$
and their weighted versions play very important roles in
analysis such as both harmonic analysis and partial differential
equations; see, for instance, \cite{gl14,s93}.

On the other hand,
along with the development of various analysis areas, there appear
many classes of function spaces which are
more inclusive and exquisite than Lebesgue spaces, such as weighted Lebesgue spaces,
Lorentz spaces, variable Lebesgue spaces, Orlicz spaces,
 Morrey spaces, and (quasi-)Banach function spaces.
For both their definitions and properties of (quasi-)Banach function spaces,
see, for instance, \cite[Chapter 1]{br} and also Definition \ref{bfs}
below. These spaces and the Hardy-type spaces based
on them have been investigated extensively
in recent decades.

It is known  that Lebesgue spaces,
Lorentz spaces, variable Lebesgue spaces, and Orlicz
spaces  are (quasi-)Banach function spaces.
However, weighted Lebesgue spaces, Herz spaces,
Morrey spaces, and Musielak--Orlicz spaces might not be
quasi-Banach function spaces (see, for instance,
\cite{st,wyyz,zyyw} for more details and examples).
Therefore, in this sense, the concept of
(quasi-)Banach function spaces is restrictive.
To establish a general framework including  also
weighted Lebesgue spaces, Herz
spaces, Morrey spaces, and Musielak--Orlicz
spaces, Sawano et al. \cite{shyy} introduced
the so-called ball quasi-Banach function space (see Definition \ref{bqbfs} below).
Moreover, based on a given  ball quasi-Banach
function space $X$, Sawano et al. \cite{shyy}  also developed a real-variable
theory of the related Hardy-type space  $H_X(\rn)$
(see Definition \ref{H} below or \cite[Definition 6.17]{shyy}), which  was originally
defined via the maximal function of Petree type.  The
equivalent characterizations  of $H_X(\rn)$, respectively, in terms of Lusin-area functions,
atoms, and radial or non-tangential maximal functions, were also established
in \cite{shyy}.
More recently, Wang et al. \cite{wyy20}
characterized the space $H_X(\rn)$
via Littlewood--Paley $g$-functions or $g_\lambda^\ast$-functions
(see also \cite{cwyz}) and obtained the boundedness of
Calder\'on--Zygmund operators on $H_X(\rn)$.
Besides, Yan et al. \cite{yyy20} established the intrinsic square
function characterizations of $H_X(\rn)$.
We also
refer the reader to \cite{dgpyyz,dlyyz,syy,syy2,yhyy,yhyyb}
for some latest progress on both ball quasi-Banach
function spaces and their related Hardy-type spaces.

A natural and interesting \emph{question} is whether or not such a
general Hardy-type space $H_X(\rn)$ can also be
characterized by the  Riesz transform, and the
main purpose of this article is to   give an affirmative answer to this question.
To be precise, let $X$ be a ball quasi-Banach function space
satisfying Assumptions \ref{a} and \ref{a2} below with both the same
$s\in(0,1]$ and $\theta\in(0,s)$. We establish the first order Riesz
transform characterization of $H_X(\mathbb{R}^n)$ when
$\theta\in[\frac{n-1}{n},s)$ and the higher order Riesz
transform characterization of $H_X(\mathbb{R}^n)$ when
$\theta\in(0,\frac{n-1}{n})$;
see Theorems \ref{thm-re} and \ref{thm-h-re} below for more details.

It is worth pointing out that
 the results obtained in this article
have a wide range of generality. More precisely, the Hardy
type space $H_X(\rn)$
considered in this article generalizes and unifies many known
Hardy-type spaces, including the classical  Hardy space,
the weighted Hardy space, the Herz--Hardy space,
the Lorentz--Hardy space,
the variable Hardy space, the mixed-norm
Hardy space, the local
generalized Herz--Hardy space, the mixed-norm Herz--Hardy
space, and the Morrey--Hardy space.
Moreover, to the best of our knowledge, even for the Herz--Hardy space,
the Lorentz--Hardy space, the mixed-norm
Hardy space, the local
generalized Herz--Hardy space, the mixed-norm Herz--Hardy space,
and the Morrey--Hardy space,
the results obtained in this article are totally  new (see Section \ref{s5} below).
We also point out that the range of
$\theta\in[\frac{n-1}{n},s)$ in Theorem \ref{thm-re}
is the best possible [see Remark \ref{re-sharp}(ii) below for more details].
These obviously reveal both the generality and the flexibility
of the main results of this article and hence more
applications are predictable.

The organization of the remainder of this article is as follows.

In Section \ref{s2}, we first recall some basic concepts of
the ball quasi-Banach function space $X$ and the
Hardy space $H_X(\rn)$ associated with $X$. We also recall the
concept of Poisson integrals and present some basic
properties of Poisson integrals of   distributions  in $H_X(\rn)$.

Section \ref{s3} is devoted to establishing the characterization of $H_X(\rn)$
via the first order Riesz transform. To   this end,
we first introduce both the Hardy space
$H_X(\rp)$ of harmonic functions associated with $X$
and the Hardy space $\mathbb{H}_X(\rp)$ of
harmonic vectors associated with $X$.
Then we establish the isomorphisms among
$H_X(\rn)$, $H_{X,2}(\rp)$, and $\mathbb{H}_{X,2}(\rp)$, where $H_{X,2}(\mathbb{R}^{n+1}_+)$
and $\mathbb{H}_{X,2}(\mathbb{R}^{n+1}_+)$ are,
respectively, certain subspaces of  $H_X(\mathbb{R}^{n+1}_+)$
and $\mathbb{H}_X(\mathbb{R}^{n+1}_+)$;
see Theorem \ref{isom} below.
Using these isomorphisms, we then obtain the first order Riesz transform
characterization of $H_X(\rn)$; see Theorem \ref{thm-re} below.
Differently from both the case of  Musielak--Orlicz--Hardy spaces
in \cite{ccyy16} and the case of variable Hardy
spaces in \cite{yzn}, since the norm of the space $H_X(\rn)$
has no explicit expression, the methods used in both
\cite{ccyy16} and \cite{yzn} are not feasible anymore.
To overcome this essential difficulty,
we assume that $X$ supports a vector-valued maximal
inequality (see Assumption \ref{a} below) and then we embed $X$
into a weighted
Lebesgue space (see Lemma \ref{emb} below).
Moreover, using this assumption, we also find
surprisingly that the parameter $\theta$ appearing
in Assumption \ref{a}, which felicitously characterizes the boundedness of
the fractional Hardy--Littlewood maximal operator on $X$,
plays the same role in some sense as that the parameter $p$ plays in $L^p(\rn)$; 
see both of (ii) and (iii) Remark \ref{re-sharp} below.

In Section \ref{s4},  we turn to establishing the higher order
Riesz transform characterization of $H_X(\rn)$.
We first recall some basic concepts associated with tensor products
and then introduce the higher order Riesz--Hardy space.
Using these spaces, we generalize Theorem \ref{thm-re} and
obtain higher order
Riesz transform characterization of $H_X(\rn)$;
see Theorem \ref{thm-h-re} below.

In Section \ref{s5}, we apply our main results obtained in Sections \ref{s3} and \ref{s4} to
four concrete examples of ball quasi-Banach
function spaces, namely, Lorentz spaces (Subsection \ref{ls}), mixed-norm
Lebesgue spaces (Subsection
\ref{mnls}), local
generalized Herz spaces (Subsection
\ref{lghs}), mixed-norm Herz spaces
(Subsection \ref{mhs}), and Morrey spaces (Subsection \ref{ms}).

Finally, we make some conventions on notation. Throughout this article,
let $\nn:=\{1,2,\dots\}$, $\zz_+:=\nn\cup\{0\}$, and $\mathbf{0}$
be the \emph{origin} of $\rn$. We always use $C$ to denote a \emph{positive constant},
independent of the main parameters involved, but perhaps varying from line to line.
Moreover, we also use $C_{(\alpha,\beta,\dots)}$ to denote a positive constant
depending on the parameters $\alpha, \beta,\dots$.
The \emph{symbol} $f\lesssim g$ means that $f\le Cg$ and, if $f\lesssim g$ and $g\lesssim f$,
we then write $f\sim g$. If $f\le Cg$ and $g=h$ or $g\le h$, we then write $f\ls g\sim h$
or $f\ls g\ls h$, \emph{rather than} $f\ls g=h$
or $f\ls g\le h$. The \emph{symbol} $\lceil s \rceil$ for any $s\in\rr$
denotes the smallest integer not less than $s$ and the \emph{symbol} $\lfloor s \rfloor$ for any $s\in\rr$
denotes the largest integer not greater than  $s$. For any subset $E$ of $\rn$,
we denote by $E^\complement$ the \emph{set} $\rn\setminus E$ and by $\mathbf{1}_{E}$
its \emph{characteristic function}. For any multi-index
$\alpha:=(\alpha_1, \dots,\alpha_n)\in\zn_+:=(\zz_+)^n$, let
$|\alpha|:= \alpha_1+\cdots+\alpha_n.$
The operator $M$ always
denotes the \emph{Hardy--Littlewood maximal operator}, which is defined by setting,
for any $f\in L^1_{\loc}(\rn)$ (the set of all locally integrable functions  on $\rn$) and $x\in\rn$,
$$M(f)(x):=\sup_{r\in(0,\infty)}\frac{1}{|B(x,r)|}\int_{B(x,r)}|f(y)|\,dy,$$
where $B(x,r)$ denotes the \emph{ball} with the center $x$ and the radius $r$.
The symbol $\varliminf$ means $\liminf$. We use $\mathcal{S}(\rn)$ to denote the \emph{Schwartz space}
equipped with the well-known classical topology determined by a countable family
of norms, while $\mathcal{S}'(\rn)$ denotes its topological dual space
equipped with the weak-$\ast$ topology.
The symbol $\mathscr{M}(\rn)$ denotes the set of
all measurable functions on $\rn$.
For any $\varphi\in\mathcal{S}(\rn)$ and $t\in(0,\infty)$,
let $\varphi_t(\cdot):= t^{-n}\varphi(\cdot/t)$. Finally, for any $q\in[1,\fz]$,
we denote by $q'$ its \emph{conjugate exponent}, namely, $1/q + 1/q'=1$. Also, when we prove a lemma,
proposition, theorem, or corollary, we always
use the same symbols in the wanted proved lemma,
proposition, theorem, or corollary.

\section{Ball Quasi-Banach Function Spaces and Poisson Integrals}\label{s2}

In this section, we recall the concepts of ball
quasi-Banach function spaces and Poisson integrals,
as well as some of their basic properties.
Let us begin with the concept of quasi-Banach
function spaces (see, for instance, \cite[Chapter 1]{br}
for more details). Recall that $\mathscr{M}(\rn)$ denotes
the set of all measurable functions on $\rn$

\begin{definition}\label{bfs}
Let $X\subset \mathscr{M}(\rn)$ be a quasi-normed linear
space equipped with a quasi-norm $\|\cdot\|_X$ which makes
sense for the whole $\mathscr{M}(\rn)$.
Then $X$ is called a \emph{quasi-Banach function space} if it satisfies
\begin{enumerate}
\item[(i)] if $f\in\mathscr{M}(\rn)$, then $\|f\|_{X}=0$
implies that $f=0$ almost everywhere;

\item[(ii)]  if $f,g\in\mathscr{M}(\rn)$, then $|g|\le|f|$ in the sense of almost everywhere implies
that $\|g\|_{X}\le\|f\|_{X}$;

\item[(iii)]  if $\{f_m\}_{m\in\nn}\subset\mathscr{M}(\rn)$
and $f\in\mathscr{M}(\rn)$, then $0\le f_{m}\uparrow f$ as $m\to\infty$
in the sense of almost
everywhere implies that $\|f_{m}\|_{X}\uparrow\|f\|_{X}$
as $m\to\infty$;

\item[(iv)]  $\mathbf{1}_E\in X$ for any measurable set $E\subset\rn$ with finite measure.
\end{enumerate}

Moreover, a quasi-Banach function space $X$ is called a
\emph{Banach function space} if it satisfies

\begin{enumerate}
\item[(v)] for any $f,g\in X$, $\|f+g\|\leq\|f\|+\|g\|$;
\item[(vi)]  for any measurable set $E\subset\rn$ with finite measure,
there exists a positive constant $C_{(E)}$, depending on $E$,
such that, for any $f\in X$,
\begin{equation*}
\int_{E} |f(x)|\,dx \le C_{(E)}\|f\|_{X}.
\end{equation*}
\end{enumerate}
\end{definition}

As is mentioned in the introduction, Lebesgue spaces,
Lorentz spaces, variable Lebesgue spaces, and Orlicz spaces  are 
(quasi-)Banach function spaces, but  weighted Lebesgue spaces, Herz spaces,
Morrey spaces, and Musielak--Orlicz spaces might not be
quasi-Banach function spaces (see, for instance,
\cite{st,wyyz,zyyw} for more details).  To give a more general framework
containing all aforementioned spaces,
Sawano et al. \cite{shyy} introduced
the following   ball quasi-Banach function spaces.

\begin{definition}\label{bqbfs}
Let $X\subset\mathscr{M}(\rn)$ be a quasi-normed linear space.
Then $X$ is called a \emph{ball quasi-Banach function space}
(for short, BQBF space) if it satisfies
(i), (ii), and (iii) of Definition \ref{bfs} and
\begin{enumerate}
\item[{\rm(vii)}] $\mathbf{1}_B\in X$ for any ball $B\subset\rn$.
\end{enumerate}

A ball quasi-Banach function space $X$
is called a \emph{ball Banach function space} if the norm of $X$ satisfies
\begin{enumerate}
\item[{\rm(viii)}] for any $f,g\in X$,
$$\|f+g\|_X\leq\|f\|_X+\|g\|_X;$$
\item[{\rm(ix)}] for any ball $B\subset\rn$, there exists a positive constant $C_{(B)}$ such that,
for any $f\in X$,
\begin{equation*}
\int_B|f(x)|\,dx\leq C_{(B)}\|f\|_X.
\end{equation*}
\end{enumerate}
\end{definition}

\begin{remark}\label{rem-ball-B}
\begin{enumerate}
\item[$\mathrm{(i)}$] Let $X$ be a ball quasi-Banach
function space on $\rn$. By \cite[Remark 2.6(i)]{yhyy} or \cite[Remark 2.5(i)]{yhyyb},
we conclude that, for any $f\in\mathscr{M}(\rn)$, $\|f\|_{X}=0$ if and only if $f=0$
almost everywhere.

\item[$\mathrm{(ii)}$] As was mentioned in
\cite[Remark 2.6(ii)]{yhyy} or \cite[Remark 2.5(ii)]{yhyyb}, we obtain an
equivalent formulation of Definition~\ref{bqbfs}
via replacing any ball $B$ by any
bounded measurable set $E$ therein.

\item[$\mathrm{(iii)}$] We should point out that,
in Definition \ref{bqbfs}, if we
replace any ball $B$ by any measurable set $E$ with
finite measure, we obtain the
definition of (quasi-)Banach function spaces which were originally
introduced in \cite[Definitions 1.1 and 1.3]{br}. Thus,
a (quasi-)Banach function space
is also a ball (quasi-)Banach function
space and the converse is not necessary to be true.
\item[$\mathrm{(iv)}$] By \cite[Theorem 2]{dfmn2021},
we conclude that both (ii) and (iii) of
Definition \ref{bqbfs} imply that any ball quasi-Banach
function space is complete and the converse is not necessary to be true.
\end{enumerate}
\end{remark}

From Definition \ref{bqbfs}, we deduce the following Fatou
property of $X$. In what follows, we denote $\lim_{k\to\infty}\inf_{j\geq k}$
simply by $\varliminf_{k\to\infty}$.

\begin{lemma}\label{fatou}
Let $X$ be a BQBF space. If both a sequence $\{f_k\}_{k\in\nn}\subset X$
and an $f\in X$ satisfy
$\varliminf_{k\to\infty}|f_k(x)|=|f(x)|$  for almost every $x\in\rn$, then
$$
\|f\|_X\leq \varliminf_{k\to\infty}\|f_k\|_X.
$$
\end{lemma}

\begin{proof}
 For any $k\in\nn$ and
$x\in\rn$, let $g_k(x):=\inf_{j\geq k}|f_j(x)|$. Then
we find that, for any $k\in\nn$ and almost every $x\in\rn$,
\begin{equation}\label{3.6x}
g_k(x)\leq |f_{k}(x)|,\ g_k(x)\leq g_{k+1}(x),
\ \text{and}\ \lim_{k\to\infty}g_k(x)=|f(x)|.
\end{equation}
From Definition \ref{bfs}(ii), we deduce that
$\|f\|_X=\|\,|f|\,\|_X$ and $\|f_k\|_X=\|\,|f_k|\,\|_X$.
This, together with \eqref{3.6x} and both (ii) and (iii)
of Definition \ref{bfs}, further implies that
$$
\|f\|_X=\|\,|f|\,\|_X=\lim_{k\to\infty}\|g_k\|
\leq\varliminf_{k\to\infty}\|\,|f_{k}|\,\|_X
=\varliminf_{k\to\infty}\|f_{k}\|_X,
$$
which completes the proof of Lemma \ref{fatou}.
\end{proof}

We now recall the concepts of both the $p$-convexification and the
convexity of $X$ (see, for instance, \cite[Chapter 2]{ors}
and \cite[Definition 1.d.3]{lt} for more details).

\begin{definition}
Let $X$ be a BQBF space and $p \in (0,\infty)$.
\begin{enumerate}
\item[{\rm(i)}] The \emph{$p$-convexification} $X^p$ of $X$ is defined by setting
$$X^p :=\lf\{f\in \mathscr{M}(\rn):\ |f|^p\in X\r\},$$
equipped with the \emph{quasi-norm} $\|f\|_{X^p}:=\|\,|f|^p\,\|_X^{1/p}$.
\item[{\rm(ii)}] The space $X$ is said to be \emph{$p$-convex}
if there exists a positive
constant $C$ such that, for any $\{f_j\}_{j\in\nn} \subset X^{1/p}$,
$$\left\|\sum_{j=1}^\infty|f_j|\right\|_{X^{1/p}}
\leq C\sum_{j=1}^\infty\|f_j\|_{X^{1/p}}.$$
In particular, when $C=1$, $X$ is said to
be \emph{strictly $p$-convex}.
\end{enumerate}
\end{definition}

Based on BQBF spaces,  Sawano et al. \cite{shyy} also introduced the following Hardy-type spaces.

\begin{definition}\label{H}
Let $X$ be a BQBF space. Let $\Phi \in \mathcal{S}(\rn)$
satisfy $\int_{\rn}\Phi(x)\,dx\neq 0$ and $b\in(0,\infty)$ be
sufficiently large. Then the \emph{Hardy space $H_X(\rn)$
associated with $X$} is defined by setting
$$H_X(\rn):=\lf\{f \in\mathcal{S}'(\rn):\ \|f\|_{H_X(\rn)}:=
\lf\|M_b^{\ast\ast}(f,\Phi)\r\|_X<\infty\r\},$$
where $M_b^{\ast\ast}(f,\Phi)$ is defined by setting,
for any $f\in \cs'(\rn)$ and $x\in\rn$,
\begin{equation*}
M_b^{\ast\ast}(f,\Phi)(x) :=\sup_{(y,t)\in\mathbb{R}_+^{n+1}}
\frac{|\Phi_t\ast f(x-y)|}{(1+t^{-1}|y|)^b}.
\end{equation*}
\end{definition}

\begin{remark}
Let all the symbols be the same as in Definition \ref{H}. Assume that there exists an $r\in(0,\infty)$ such that the Hardy--Littlewood maximal operator $M$ is bounded on $X^{1/r}$. 
If $b\in(n/r,\infty)$, then, by \cite[Theorem 3.1]{shyy}, we find 
that the Hardy space $H_X(\rn)$ is independent of the choice of $b$.
\end{remark}

Denote by $L^{1}_{\rm loc}(\rn)$ the set of all locally integrable functions
on $\rn$. For any $\theta\in(0,\infty)$, the \emph{powered Hardy--Littlewood maximal operator}
$M^{(\theta)}$ is defined by setting,
for any $f\in L^{1}_{\rm loc}(\rn)$ and $x\in\rn$,
\begin{equation*}
M^{(\theta)}(f)(x):=\lf\{M\lf(|f|^\theta\r)(x)\r\}^{\frac{1}{\theta}}.
\end{equation*}

Moreover, we also need some basic assumptions on $X$ as follows
(see also \cite[(2.8) and (2.9)]{shyy}).

\begin{assumption}\label{a}
Let $X$ be a BQBF space. Assume that, for
some $\theta,s\in (0,1]$ such that  $\theta<s$,
there exists a positive constant $C$ such that, for
any $\{f_j\}_{j=1}^\infty \subset L^{1}_{\rm loc}(\rn)$,
\begin{equation}\label{ma}
\left\|\left\{\sum_{j=1}^\infty\left[M^{(\theta)}
(f_j)\right]^s\right\}^{\frac{1}{s}}\right\|_X\leq C
\left\|\left\{\sum_{j=1}^\infty|f_j|^s\right\}^{\frac{1}{s}}\right\|_X.
\end{equation}
\end{assumption}

\begin{remark}\label{r2.1}
The inequality \eqref{ma} is called the \emph{Fefferman--Stein
vector-valued maximal inequality} on $X$.
If $X:=L^p(\rn)$ with $p\in(1,\infty)$, $\theta=1$, and $s\in(1,\infty]$, the inequality \eqref{ma}
was originally established by Fefferman and Stein \cite[Theorem 1]{fs2}.
See, for instance, \cite[Remark 2.4]{wyy20} for some examples of ball quasi-Banach
function spaces satisfying \eqref{ma}.
\end{remark}

To state the next assumption on $X$, we need the concept of the associate space.
For any ball Banach function space $X$, the \emph{associate space}
(also called \emph{K\"{o}the dual}) $X'$ of $X$ is defined by setting
$$X':=\lf\{f \in \mathscr{M}(\rn):\ \|f\|_{X'}<\infty\r\},
$$
where, for any $f \in \mathscr{M}(\rn)$,
$$
\|f\|_{X'}:=
\sup\{\|fg\|_{L^1(\rn)}:\ g \in X,\ \|g\|_X=1\}
$$
(see, for instance, \cite[Chapter 1, Section 2]{br} for the details).
Recall that, for any given ball Banach function space $X$, $X'$ is also a
ball Banach function space (see \cite[Proposition 2.3]{shyy}).

\begin{assumption}\label{a2}
Let $X$ be a BQBF space. Assume that there exists an $s\in(0,1]$
such that $X^{1/s}$ is also a ball Banach function space and there exists a $q\in(1,\fz]$
and a $C\in(0,\fz)$ such that, for any $f\in (X^{1/s})'$,
\begin{equation}\label{ma21}
\left\|M^{((q/s)')}(f)\right\|_{(X^{1/s})'}\le  C \|f\|_{(X^{1/s})'}.
\end{equation}
\end{assumption}

\begin{remark}
We point out that, in \cite[Theorems 2.10, 3.7 and 3.21]{shyy}, one needs the additional assumption
that there exists an $s\in(0,1]$ such that $X^{1/s}$ is a ball Banach function space.
Indeed, this assumption ensures that $(X^{1/s})'$ is
also a ball Banach function space, which further implies that,
for any $f\in (X^{1/s})'$,  $f\in L^1_{\loc}(\rn)$ and hence the Hardy-Littlewood maximal
operator can be defined on $(X^{1/s})'$.
\end{remark}

Next, we recall some basic concepts associated with the Poisson integral.

\begin{definition}\label{poisson}
\begin{enumerate}
\item[{\rm(i)}] A distribution $f\in\cs'(\rn)$ is called
a \emph{bounded distribution} if, for any $\phi\in\cs(\rn)$,
$\phi\ast f\in L^\infty(\rn)$, where
$$L^\infty(\rn):=\left\{f \in \mathscr{M}(\rn):\
\|f\|_{L^\infty(\rn)}:=\esssup_{x\in\rn}|f(x)|<\infty\right\}.$$
\item[{\rm(ii)}] For any $(x,t)\in\rp$,
\begin{equation*}
P_t(x):=c_{(n)}\frac{t}{(t^2+|x|^2)^{(n+1)/2}}
\end{equation*}
is called the \emph{Poisson kernel}, where $c_{(n)}$
is the same as in \eqref{unit-b}.
\item[{\rm(iii)}] Assume that $f\in\cs'(\rn)$ is a bounded distribution.
The \emph{non-tangential maximal function} $M(f;P)$
of $f$ is defined by setting,
for any $x\in\rn$,
\begin{equation}\label{eq-max-p}
M(f;P)(x):=\sup_{\{(y,t)\in\rp:\ |y-x|<t\}}|P_t\ast f(y)|.
\end{equation}
\end{enumerate}
\end{definition}

\begin{remark}\label{re-ptf}
If $f\in L^p(\rn)$ with $p\in[1,\infty]$,
then, by the Young inequality, it is easy to show
that $f$ is a bounded distribution. Moreover,
if $f$ is a bounded distribution, then $P_t\ast f$ is
a well-defined, bounded, and smooth  harmonic
function on $\rp$ (see, for instance, \cite[p.\,90]{s93}).
\end{remark}

The following theorem establishes the Poisson integral
characterization of a bounded distribution in $H_X(\rn)$.

\begin{theorem}\label{thm-p}
Let $X$ be a BQBF space such that $M$ is bounded on $X^r$ for some $r\in(0,\infty)$.
Assume that $f\in\cs'(\rn)$ is a bounded distribution.
Then $f\in H_X(\rn)$ if and only if $M(f;P)\in X$, where
$M(f;P)$ is the same as in \eqref{eq-max-p}.
Moreover, there exist positive constants $C_1$ and $C_2$, independent
of $f$, such that
$$
C_1\|f\|_{H_X(\rn)}\leq\|M(f;P)\|_{X}\leq C_2\|f\|_{H_X(\rn)}.
$$
\end{theorem}

\begin{proof}
We first prove the necessity. To this end, fix a bounded
distribution $f\in\cs'(\rn)\cap H_X(\rn)$.
For any $N\in\nn$ and $\phi\in \cs(\rn)$, let
$$
p_N(\phi):=\sum_{\alpha\in\zn_+,\, |\alpha|<N}
\sup_{x\in\rn}(1+|x|)^N|\partial^\alpha\phi(x)|
$$
and
$$
\cf_N:=\left\{\phi\in\cs(\rn):\ p_N(\phi)\leq 1\right\}.
$$
For any $f\in \cs'(\rn)$, the \emph{grand maximal function}
$M_N(f)$ is defined by setting, for any $x\in\rn$,
$$
M_N(f)(x):=\sup\left\{|\phi_t\ast f(y)|:\ (x,t)\in\rp,\ |x-y|<t,\ \phi\in\cf_N\right\}.
$$
By \cite[Theorem 3.1(i)]{shyy}, we find that there exists an
$N\in\nn$ such that
\begin{equation}\label{eq-thm-p1}
\|M_N(f)\|_{X}\lesssim\|f\|_{H_X(\rn)}<\infty.
\end{equation}
On the other hand, from an argument similar to that used in the
estimation of \cite[(2.1.39)]{gl14}, we deduce that
$$
M(f;P)\lesssim M_N(f).
$$
Combining this, Definition \ref{bfs}(ii), and \eqref{eq-thm-p1},
we obtain
$$
\|M(f;P)\|_X\lesssim \|M_N(f)\|_X\lesssim\|f\|_{H_X(\rn)}<\infty.
$$
This finishes the proof of the necessity.

Next, we show the sufficiency. To this end, fix a bounded
distribution $f\in\cs'(\rn)$ such that $\|M(f;P)\|_X<\infty$.
By the same argument as that used in \cite[p.\,99]{s93}, we know that there exists
a $\Phi\in\cs(\rn)$, satisfying $\int_{\rn}\Phi(x)\,dx\neq 0$,
such that, for any $x\in\rn$,
$$
M(f;\Phi)(x)\lesssim \sup_{t\in(0,\infty)}|P_t\ast f(x)| \lesssim M(f;P)(x),
$$
where
\begin{equation}\label{2.6x}
M(f;\Phi)(x):=\sup_{t\in(0,\infty)}|\Phi_t\ast f(x)|.
\end{equation}
This, together with both \cite[Theorem 3.1(ii)]{shyy} and
Definition \ref{bfs}(ii), further implies that
$$
\|f\|_{H_X(\rn)}\lesssim\|M(f;\Phi)\|_X\lesssim\|M(f;P)\|_X<\infty,
$$
which completes the proof of the sufficiency and hence
Theorem \ref{thm-p}.
\end{proof}

\begin{remark}
In \cite[Theorem 3.3]{shyy}, Sawano et al. also
established the Poisson integral characterization of $H_X(\rn)$.
 We point out that Theorem
\ref{thm-p} is different from \cite[Theorem 3.3]{shyy} in the following sense.
\begin{enumerate}
\item[{\rm(i)}]
In \cite[Theorem 3.3]{shyy}, Sawano et al.  used
the quantity $\sup_{t\in(0,\infty)}|P_t\ast f|$ instead of $M(f;P)$.
\item[{\rm(ii)}] In \cite[Theorem 3.3]{shyy}, Sawano et al. needed an additional assumption
that there exists a positive constant
$C$ such that
\begin{equation}\label{u-low}
\inf_{x\in\rn}\|\mathbf{1}_{B(x,1)}\|_X\geq C.
\end{equation}
This is a very strong condition because \eqref{u-low}
does not hold true even when $X$ is some weighted Lebesgue space.
\item[{\rm(iii)}] In Theorem \ref{thm-p}, we assume, a priori,
that $f\in\cs'(\rn)$ is a bounded distribution. This is because,
for an arbitrary $f\in H_X(\rn)$, we cannot show that $f$
is a bounded distribution without any additional assumptions
(see also \cite[Remark 2.5]{ccyy16}).
\end{enumerate}
\end{remark}

Now, we introduce the concept of the Poisson integral $P_t\ast f$,
where $f\in \cs'(\rn)$ is a limit in $\cs'(\rn)$ of a
sequence $\{f_k\}_{k\in\nn}\subset L^2(\rn)$.

\begin{definition}\label{def-pi-o}
For any $f\in \cs'(\rn)$ and
$\{f_k\}_{k\in\nn}\subset L^2(\rn)$
satisfying $\lim_{k\to\infty}f_k=f$ in $\cs'(\rn)$
and for any $(x,t)\in\rp$, define
\begin{equation}\label{def-pi}
P_t\ast f(x):=\lim_{k\to\infty} P_t\ast f_k(x)
\end{equation}
pointwisely.
\end{definition}

The following lemma indicates
that \eqref{def-pi} is well defined for any $f\in\cs'(\rn)$.

\begin{lemma}\label{d-n}
Let $f\in \cs'(\rn)$. Then the following statements hold true.
\begin{enumerate}
\item[{\rm(i)}] There exists a sequence $\{f_k\}_{k\in\nn}\subset C_{\mathrm c}^\infty(\rn)$ such that
$\lim_{k\to\infty}f_k=f$ in $\cs'(\rn)$.
\item[{\rm(ii)}] If $\{f_k\}_{k\in\nn}\subset L^2(\rn)$
satisfies $\lim_{k\to\infty}f_k=f$ in $\cs'(\rn)$,
then, for any $(x,t)\in\rp$,
$\lim_{k\to\infty}P_t\ast f_k(x)$ exists.
\item[{\rm(iii)}] If
$\{f_k\}_{k\in\nn},\{g_k\}_{k\in\nn}\subset L^2(\rn)$
satisfies that
$\lim_{k\to\infty}f_k=f=\lim_{k\to\infty}g_k$ in $\cs'(\rn),$
then, for any $(x,t)\in\rp$,
\begin{equation}\label{eq-pfk-pgk}
\lim_{k\to\infty}P_t\ast f_k(x)=\lim_{k\to\infty}P_t\ast g_k(x).
\end{equation}
\end{enumerate}
\end{lemma}

\begin{proof}
(i) is just \cite[Proposition 2.3.23]{gl}; we omit details here.
To show (ii), fix both an $f\in \cs'(\rn)$ and a
sequence $\{f_k\}_{k\in\nn}\subset L^2(\rn)$
such that $\lim_{k\to\infty}f_k=f$ in $\cs'(\rn)$.
Notice that, by \cite[p.\,90]{s93}, we find that
there exist $\phi,\psi\in\cs(\rn)$ and $h\in L^1(\rn)$ such that, for any $(x,t)\in\rp$,
$$
P_t(x)= h_t\ast \phi_t(x)+\psi_t(x).
$$
Thus, for any $k,j\in\nn$ and $(x,t)\in\rp$, we have
\begin{align*}
\left|P_t\ast f_k(x)-P_t\ast f_j(x)\right|
&=\left| h_t\ast \phi_t\ast(f_k-f_j)(x)+\psi_t\ast(f_k-f_j)(x)\right|.
\end{align*}
Since $\{f_k\}_{k\in\nn}$ converges to $f$ in $\cs'(\rn)$,
it follows that, for any $\epsilon\in(0,\infty)$ and
$\varphi\in\cs(\rn)$, there exists
a $K\in\nn$ such that, for any $k>K$ and $j>K$,
$$
|\langle f_k-f_j,\varphi\rangle|<\epsilon.
$$
Thus, for any $(x,t)\in\rp$, we can find a $K_{(x,t)}\in \nn$
such that, for any $k>K_{(x,t)}$ and $j>K_{(x,t)}$,
$$
|\phi_t\ast (f_k-f_j)(x)|<\epsilon
\ \text{and}\
|\psi_t\ast (f_k-f_j)(x)|<\epsilon
$$
and hence
\begin{equation}\label{eq-pfkfj}
\left|P_t\ast f_k(x)-P_t\ast f_j(x)\right|<
(|h_t|\ast \epsilon)(x) +\epsilon=\left[\|h\|_{L^1(\rn)}+1\right]\epsilon.
\end{equation}
This shows that $\{P_t\ast f_k(x)\}_{k\in\nn}$ is a
Cauchy sequence for any $(x,t)\in\rp$. Therefore,
$$\lim_{k\to\infty}P_t\ast f_k(x)$$ exists, which
completes the proof of (ii).

Finally, we show (iii). To this end, by an argument
similar to that used in the estimation of \eqref{eq-pfkfj}
via $f_j$ replaced by $g_k$,
we find that, for any $\epsilon\in(0,\infty)$ and $(x,t)\in\rp$,
there exists a $\widetilde{K}_{(x,t)}\in(0,\infty)$ such that,
for any $k>\widetilde{K}_{(x,t)}$,
$$
\left|P_t\ast f_k(x)-P_t\ast g_k(x)\right|<
\left[\|h\|_{L^1(\rn)}+1\right]\epsilon,
$$
which implies that \eqref{eq-pfk-pgk} holds true.
This finishes the proof of (iii) and hence Lemma \ref{d-n}.
\end{proof}

Next, we recall the concept of absolutely continuous quasi-norms as follows
(see, for instance, \cite[Definition 3.1]{br}).

\begin{definition}
Let $X$ be a BQBF space. A function $f\in X$ is said to have an
\emph{absolutely continuous quasi-norm} in $X$
if $\|f\mathbf{1}_{E_j}\|_X\downarrow 0$ as $j\to\infty$
whenever $\{E_j\}_{j=1}^\infty$ is a sequence of measurable sets that
satisfy $E_j \supset E_{j+1}$ for any $j \in \nn$
and $\bigcap_{j=1}^\infty E_j = \emptyset$.
Moreover, $X$ is said to have an \emph{absolutely
continuous quasi-norm} if, for any $f\in X$, $f$  has an
absolutely continuous quasi-norm in $X$.
\end{definition}

\begin{remark}
We point out that
many function spaces  such as Lebesgue spaces, Lorentz spaces, weighted
Lebesgue spaces, Herz spaces,
variable exponent Lebesgue spaces, and Orlicz-slice space,
have absolutely continuous quasi-norms, but  the Morrey space might have
no absolutely continuous quasi-norm (see
\cite[p.\,10]{shyy} and \cite[Remark 3.4]{wyy20} for more details).
\end{remark}

The following conclusion is a combination of
\cite[Theorems 3.6 and 3.7 and Remark 3.12]{shyy}; we omit the details here.

\begin{lemma}\label{den-h}
Assume that $X$ is a BQBF space and satisfies both Assumptions \ref{a} and
\ref{a2}.
If $f\in H_X(\rn)$, then there exists a sequence
$\{f_k\}_{k\in\nn}\subset L^2(\rn)\cap H_X(\rn)$
and a positive constant $C$ such that
\begin{equation}\label{fk-f}
\lim_{k\to\infty} f_k=f
\end{equation}
in $\cs'(\rn)$ and, for any $k\in\nn$, $\|f_k\|_{H_X(\rn)}\leq C\|f\|_{H_X(\rn)}$. Moreover, if $X$ has an absolutely
continuous quasi-norm, then  \eqref{fk-f} holds true in $H_X(\rn)$.
\end{lemma}

From both Definition \ref{def-pi-o} and Lemma \ref{den-h},
we deduce that, if $X$ is a BQBF space and  satisfies both Assumptions \ref{a} and
\ref{a2}, then, for
any $f\in H_X(\rn)$, $P_t\ast f$ is well defined.
Moreover, if $X$ has an absolutely continuous quasi-norm,
then we have the following conclusion.

\begin{proposition}\label{well-def}
Assume that $X$ is a BQBF space, satisfies both Assumptions \ref{a} and
\ref{a2}, and has an absolutely continuous
quasi-norm. Then, for any compact set $K\subset \rp$, \eqref{def-pi}
converges uniformly on $K$.
\end{proposition}

To prove Proposition \ref{well-def}, we need  the
following concept of weighted Lebesgue spaces.

\begin{definition}
For any $q\in[1,\fz]$, denote by $A_q(\rn)$ the class of all \emph{Muckenhoupt weights} (see, for instance,
\cite[Definitions 7.1.1 and 7.1.3]{gl} for its definition). For any $p\in(0,\fz)$ and $w\in A_\fz(\rn)$,
the \emph{weighted Lebesgue space} $L^p_w(\rn)$ is defined by setting
$$L^p_w(\rn):=\left\{f \in \mathscr{M}(\rn):\
\|f\|_{L^p_w(\rn)}:=\left[\int_{\rn}|f(x)|^pw(x)\,dx\right]^{1/p}<\infty\right\}.$$
\end{definition}

The following lemma is just \cite[Lemma 4.7]{cwyz}.

\begin{lemma}\label{emb}
Let $X$ be a BQBF space. Assume that
there exists an $s\in(0, \infty)$ such that $X^{1/s}$ is
a ball Banach function space and $M$ is bounded on $(X^{1/s})'$.
Then there exists an $\varepsilon\in(0, 1)$
such that $X$ continuously embeds into $L^s_w(\rn)$
with $w:= [M(\mathbf{1}_{B(\mathbf{0},1)})]^\varepsilon$.
\end{lemma}

\begin{remark}\label{def-w}
\begin{enumerate}
\item[{\rm(i)}] Let $X$ be a BQBF space satisfying Assumption \ref{a2}.
Then, by both \cite[Lemma 2.10]{cwyz} and \eqref{ma21},
we find that $M$ is bounded on $(X^{1/s})'$.
\item[{\rm(ii)}]Let $w$ be the same as in Lemma \ref{emb}.
By \cite[Theorem 7.2.7]{gl}, we know that $w\in A_1(\rn)$.
For any measurable set $E\subset \rn$, let
\begin{equation}\label{mu}
\mu(E):=\int_Ew(x)\,dx.
\end{equation}
From \cite[Proposition 7.1.5]{gl}, it follows that $\mu$
is a \emph{doubling measure}.
For any $p\in(0,\fz)$, let
$$
L^p(\mu):=\left\{f \in \mathscr{M}(\rn):\
\|f\|_{L^p(\mu)}:=\left[\int_{\rn}|f(x)|^p\,d\mu(x)\right]^{1/p}<\infty\right\}.
$$
It is easy to show that $L^p(\mu)=L^p_w(\rn)$  with equivalent
quasi-norms.
\end{enumerate}
\end{remark}

Next, we show Proposition \ref{well-def}.

\begin{proof}[Proof of Proposition \ref{well-def}]
To show this proposition, fix an $(x_0,t_0)\in\rp$
and an $f\in H_X(\rn)$. By Lemma \ref{den-h}, we
find a sequence $\{f_k\}_{k\in\nn}\subset
L^2(\rn)\cap H_X(\rn)$ such that
$\lim_{k\to\infty} f_k =f$ in both $H_X(\rn)$ and $\cs'(\rn)$.
We first claim that \eqref{def-pi} converges uniformly on
$B(x_0,t_0/4)\times(3t_0/4,5t_0/4)$.
Indeed, for any $x,y\in B(x_0,t_0/4)$ and
$t\in(3t_0/4,5t_0/4)$, we have
$$
|x-y|\le|x-x_0|+|x_0-y|<\frac{t_0}{2}<t.
$$
From this, we deduce that, for any $x,y\in B(x_0,t_0/4)$,
$t\in(3t_0/4,5t_0/4)$, and $k,j\in\nn$,
\begin{align}\label{eq-p-1}
&\left|P_t\ast f_k(x)-P_t\ast f_j(x)\right|\\
&\quad=\left|P_t\ast (f_k-f_j)(x)\right|\leq \sup_{\{(z,t)\in\rp:\ |z-y|<t\}}
\left|P_t\ast (f_k-f_j)(z)\right|\notag\\
&\quad= M(f_k-f_j;P)(y),\notag
\end{align}
where $M(f_k-f_j;P)$ is the same as in \eqref{eq-max-p}
via replacing $f$ by $f_k-f_j$.
For any $k,j\in\nn$, let
$$
E_{k,j}:=\left\{y\in B\left(x_0,\frac{t_0}{4}\right):\
M(f_k-f_j;P)(y)\leq 1\right\}.
$$
By this, \eqref{mu}, Lemma \ref{emb}, Theorem \ref{thm-p},
and the fact that $\{f_k\}_{k\in\nn}$
is a Cauchy sequence in $H_X(\rn)$, we conclude that
\begin{align*}
\mu\left(B\left(x_0,\frac{t_0}{4}\right)\setminus E_{k,j}\right)
&=\int_{B(x_0,t_0/4)\setminus E_{k,j}}w(y)\,dy
\leq \int_{\rn}\left[M(f_k-f_j;P)(y)\right]^sw(y)\,dy\\
&=\left\|M(f_k-f_j;P)\right\|_{L^s_w(\rn)}^s\lesssim \left\|M(f_k-f_j;P)\right\|_{X}^s\\
&\lesssim \left\|f_k-f_j\right\|_{H_X(\rn)}^s\to 0
\end{align*}
as $k,j\to\infty$, where $w$ and $s$ are the same as in Lemma \ref{emb}.
This implies that there exists a positive
constant $\widetilde{L}$, depending on both $x_0$ and $t_0$, such that,
for any $k>\widetilde{L}$ and $j>\widetilde{L}$,
$$
\mu\left(B\left(x_0,\frac{t_0}{4}\right)\setminus E_{k,j}\right)
<\frac{1}{2}\mu\left(B\left(x_0,\frac{t_0}{4}\right)\right)
$$
and hence
\begin{equation}\label{mu-ekj}
\mu\left(E_{k,j}\right)
\ge \frac{1}{2}\mu\left(B\left(x_0,\frac{t_0}{4}\right)\right).
\end{equation}
Since $\{f_k\}_{k\in\nn}$
is a Cauchy sequence in $H_X(\rn)$, it follows that,
for any $\epsilon\in(0,\infty)$, there exists an
$\widetilde{L}_0\in(\widetilde{L},\infty)$ such that,
for any $k>\widetilde{L}_0$ and $j>\widetilde{L}_0$,
$$
\|f_k-f_j\|_{H_X(\rn)}<\epsilon.
$$
Using this, \eqref{eq-p-1}, \eqref{mu-ekj}, Lemma \ref{emb},
and Theorem \ref{thm-p}, we find that,
for any $x\in B(x_0,t_0/4)$,   $t\in(3t_0/4,5t_0/4)$,
$k>\widetilde{L}_0$, and $j>\widetilde{L}_0$,
\begin{align}\label{2.10x}
&\left|P_t\ast f_k(x)-P_t\ast f_j(x)\right|\\
&\quad=\left\{\frac{1}{\mu(E_{k,j})}\int_{E_{k,j}}\left|P_t\ast f_k(x)-P_t\ast f_j(x)\right|^sw(y)\,dy\right\}^{1/s}\notag\\
&\quad\leq\left\{\frac{1}{\mu(E_{k,j})}\int_{E_{k,j}}\left|M(f_k-f_j;P)(y)\right|^sw(y)\,dy\right\}^{1/s}\notag\\
&\quad\lesssim \|M(f_k-f_j;P)\|_{L^s_w(\rn)}
\lesssim\|M(f_k-f_j;P)\|_{X}\lesssim
\|f_k-f_j\|_{H_X(\rn)}\lesssim \epsilon,\notag
\end{align}
where the implicit positive constants may depend on both $x_0$ and $t_0$.
This, together with \cite[Theorem 7.8]{r76}, implies that \eqref{def-pi}
converges uniformly on
$B(x_0,t_0/4)\times(3t_0/4,5t_0/4)$ and
hence finishes the proof of the above claim.

Now, for any compact set $K\subset \rp$, from
$$
K\subset \bigcup_{(x,t)\in K} \left[B\left(x,\frac{t}{4}\right)
\times\left(\frac{3t}{4},\frac{5t}{4}\right)\right],
$$
it follows that there exists an $m\in\nn$ and finite
points $\{(x_i,t_i)\}_{i=0}^m\subset K$ such that
$$
K\subset \bigcup_{i=0}^m \left[B\left(x_i,\frac{t_i}{4}\right)
\times\left(\frac{3t_i}{4},\frac{5t_i}{4}\right)\right].
$$
By the above claim, for any $\epsilon\in(0,\infty)$
and $i\in\{1,\dots,m\}$, we can find an $L_i\in(0,\infty)$
such that, for any $(x,t)\in B(x_i,t_i/4)\times(3t_i/4,5t_i/4)$,
$k>L_i$, and $j>L_i$,
\begin{equation}\label{eq-pfkj}
\left|P_t\ast f_k(x)-P_t\ast f_j(x)\right|<\epsilon.
\end{equation}
Let $L:=\max\{L_1,\dots,L_m\}$. Then,
for any $(x,t)\in K$, $k>L$, and $j>L$, \eqref{eq-pfkj} holds true.
From this and \cite[Theorem 7.8]{r76},
we deduce that \eqref{def-pi}
converges uniformly on $K$. This finishes the proof
of Proposition \ref{well-def}.
\end{proof}

Using both Lemma \ref{d-n} and Proposition \ref{well-def},
we obtain the following conclusions.

\begin{corollary}\label{co-pt}
Assume that $X$ is a BQBF space and satisfies both Assumptions \ref{a} and
\ref{a2}.
\begin{enumerate}
\item[{\rm (i)}] Then, for any $(x,t)\in\rp$ and $f,g\in H_X(\rn)$,
$P_t\ast (f+g)(x)=P_t\ast f(x)+P_t\ast g(x).$
\item[{\rm (ii)}] There exists a positive constant
$C$ such that, for any $t\in(0,\infty)$ and $f\in H_X(\rn)$,
$$\|P_t\ast f\|_{X}\leq C \|f\|_{H_X(\rn)}.$$
\item[{\rm (iii)}] If $X$ has an absolutely continuous
quasi-norm, then, for any $f\in H_X(\rn)$,
$u(x,t):=P_t\ast f(x)$ is harmonic on $\rp$.
\end{enumerate}
\end{corollary}

\begin{proof}
We first prove (i). By Lemma \ref{d-n}(i), we find  two
sequences $\{f_k\}_{k\in\nn},\{g_k\}_{k\in\nn}\subset C_{\mathrm c}^\infty(\rn)$
such that
\begin{equation}\label{fkgk}
\lim_{k\to\infty} f_k=f\ \text{and}\ \lim_{k\to\infty} g_k=g
\end{equation}
in $\cs'(\rn)$, which further imply that
$$
\lim_{k\to\infty} (f_k+g_k)=f+g
$$
in $\cs'(\rn)$. From this, Definition \ref{def-pi-o}, and \eqref{fkgk},
we infer that,  for any $(x,t)\in\rp$,
\begin{align*}
P_t\ast (f+g)(x)&=\lim_{k\to\infty} P_t\ast(f_k+g_k)(x)\\
& =\lim_{k\to\infty} P_t\ast f_k(x)+\lim_{k\to\infty} P_t\ast g_k(x) =P_t\ast f(x)+P_t\ast g(x).
\end{align*}
This finishes the proof of (i).

Next, we show (ii). To this end, fix an $f\in H_X(\rn)$.
By Lemma \ref{den-h}, we  choose a sequence
$\{f_k\}_{k\in\nn}\subset L^2(\rn)$ such that
$\lim_{k\to\infty} f_k=f$ in $\cs'(\rn)$ and
$$
\|f_k\|_{H_X(\rn)}\lesssim\|f\|_{H_X(\rn)}.
$$
From this, Definition \ref{def-pi-o}, Definition \ref{bfs}(ii),
Lemma \ref{fatou}, and Theorem \ref{thm-p}, we deduce that, for any $t\in(0,\infty)$,
\begin{align*}
\|P_t\ast f\|_{X}&=\left\|\lim_{k\to\infty} P_t\ast f_k\right\|_{X}
=\left\|\,\left|\lim_{k\to\infty} P_t\ast f_k\,\right|\,\right\|_{X}\\
&=\left\|\lim_{k\to\infty} \left|P_t\ast f_k\right|\right\|_{X}
\leq \varliminf_{k\to\infty}\left\|\,\left|P_t\ast f_k\right|\,\right\|_{X}\\
&\leq \varliminf_{k\to\infty}\left\|M(f_k;P)\right\|_{X}\lesssim \varliminf_{k\to\infty}\|f_k\|_{H_X(\rn)}\lesssim \|f\|_{H_X(\rn)},
\end{align*}
which completes the proof of (ii).

Finally, we show (iii). To this end, fix an $f\in H_X(\rn)$.
By Lemma \ref{den-h}, we  choose a sequence
$\{f_k\}_{k\in\nn}\subset L^2(\rn)$ such that
$\lim_{k\to\infty} f_k=f$ in both $\cs'(\rn)$ and $H_X(\rn)$.
Using Definition \ref{def-pi-o}, we find that, for any $(x,t)\in\rp$,
\begin{equation*}
u(x,t)=\lim_{k\to\infty} P_t\ast f_k(x).
\end{equation*}
For any $k\in\nn$, from both $f_k\in L^2(\rn)$ and Remark \ref{re-ptf},
we infer that $u_k(x,t):=P_t\ast f_k(x)$ is harmonic on $\rp$.
Using this, Proposition \ref{well-def}, and \cite[p.\,42,
Corollary 1.8]{sw71}, we conclude that $u$ is harmonic on $\rp$.
This finishes the proof of (iii) and hence Corollary \ref{co-pt}.
\end{proof}

\section{First Order Riesz Transform Characterization}\label{s3}

In this section, we establish the first order Riesz transform
characterization of $H_X(\rn)$. In order to achieve
this goal, we need to introduce both the Hardy space $H_X(\rp)$ of
harmonic functions and the Hardy space
$\bh_X(\rp)$ of harmonic vectors on the
upper half space $\rp$, and clarify their relations with $H_X(\rn)$.
Let us begin with the concept of the Hardy type space $H^p(\rp)$ of
harmonic functions.
Recall that a function $u$ on $\rp$ is said to be \emph{harmonic} if,
for any $(x,t)\in\rp$,
$$
\sum_{j=1}^n\frac{\partial^2}{\partial x_j^2}u(x,t)+\frac{\partial^2}{\partial t^2}u(x,t)=0.
$$

\begin{definition}\label{def-h-p}
Let $p\in(0,\infty)$. The \emph{Hardy space $H^p(\rp)$ of
harmonic functions} is defined to be the set of all the harmonic
functions $u$ on $\rp$ such that
$$
\|u\|_{H^p(\rp)}:= \sup_{t\in(0,\infty)}\|u(\cdot,t)\|_{L^p(\rn)}<\infty.
$$
\end{definition}

Now, we introduce the concept of Hardy  spaces $H_X(\rp)$ of
harmonic functions associated with a BQBF space $X$.

\begin{definition}\label{def-h-x}
Let $X$ be a BQBF space.
\begin{enumerate}
\item[{\rm(i)}] Let $u$ be a function on
$\rp$. Its \emph{nontangential maximal function} $u^\ast$ is
defined by setting,  for any $x\in\rn$,
$$
u^\ast(x):=\sup_{\{(y,t)\in\rp:\ |y-x|<t\}}|u(y,t)|.
$$
\item[{\rm(ii)}] The \emph{Hardy space $H_X(\rp)$ of
harmonic functions associated with $X$} is defined to
be the set of all the harmonic functions $u$ on $\rp$ such that
$u^\ast\in X$.
Moreover, for any $u\in H_X(\rp)$, its (quasi-)norm
$\|u\|_{H_X(\rp)}$ is defined by setting
$$
\|u\|_{H_X(\rp)}:= \|u^\ast\|_{X}.
$$
\item[{\rm(iii)}] The subspace $H_{X,2}(\rp)$ is defined
to be the set of all the functions $u\in H_X(\rp)$ satisfying
that there exists a sequence $\{u_k\}_{k\in\nn}
\subset H_X(\rp)\cap H^2(\rp)$ such that
$$
u=\lim_{k\to\infty} u_k
$$
in $H_X(\rp)$.  Moreover, for any $u\in H_{X,2}(\rp)$,
let $\|u\|_{H_{X,2}(\rp)}:=\|u\|_{H_{X}(\rp)}$.
\end{enumerate}
\end{definition}

The following proposition establishes the relation
between $H_X(\rn)$ and $H_{X,2}(\rp)$ via the Poisson integral.

\begin{proposition}\label{lem-hx-hx2}
Assume that $X$ is a BQBF space, satisfies both Assumptions \ref{a} and
\ref{a2}, and has an absolutely continuous quasi-norm.
Let $u$ be a harmonic function on $\rp$. Then $u\in H_{X,2}(\rp)$
if and only if there exists an $f\in H_X(\rn)$ such that,
for any $(x,t)\in\rp$, $u(x,t)=P_t\ast f(x)$,
where $P_t\ast f$ is the same as in \eqref{def-pi};
moreover, there exist positive constants $C_1$ and $C_2$,
independent of both $f$ and $u$, such that
$$
C_1\|f\|_{H_X(\rn)}\leq \|u\|_{H_X(\rp)}\leq C_2\|f\|_{H_X(\rn)}.
$$
\end{proposition}

To prove Proposition \ref{lem-hx-hx2}, we need
the following lemma on the convergence.

\begin{lemma}\label{p-n}
Let $X$ be a BQBF space. Assume that
there exists an $s\in(0, \infty)$ such that $X^{1/s}$ is
a ball Banach function space and $M$ is bounded on $(X^{1/s})'$.
If both $\{f_k\}_{k\in\nn}\subset H_X(\rp)$ and $f\in H_X(\rp)$
satisfy $\lim_{k\to\infty} f_k=f$ in $H_X(\rp)$, and if
$\{\epsilon_k\}_{k\in\nn}\subset(0,\infty)$ satisfies $\lim_{k\to\infty}\epsilon_k=0$,
then, for any $(x,t)\in\rp$,
\begin{equation}\label{eq-p-n}
\lim_{k\to\infty} f_k(x,t+\epsilon_k)=f(x,t).
\end{equation}
\end{lemma}

\begin{proof}
Assume that both $\{f_k\}_{k\in\nn}\subset H_X(\rp)$ and $f\in H_X(\rp)$
satisfy
\begin{equation}\label{eq-pr-pn}
\lim_{k\to\infty} f_k=f
\end{equation}
in $H_X(\rp)$.
Since $f\in H_X(\rp)$, it follows that $f$ is harmonic.
Thus, to prove \eqref{eq-p-n}, it suffices to show that, for any $(x,t)\in\rp$,
\begin{equation}\label{3.1x}
\lim_{k\to\infty} f_k(x,t+\epsilon_k)
=\lim_{k\to\infty} f(x,t+\epsilon_k).
\end{equation}
Observe that,
for any $k\in\nn$, $(x,t)\in\rp$, and $y\in B(x,t)$,
\begin{align*}
|f_k(x,t+\epsilon_k)-f(x,t+\epsilon_k)|
&\leq \sup_{\{(z,s)\in\rp:\ |z-y|<s\}}|f_k(z,s)-f(z,s)|\\
&=(f_k-f)^\ast(y).
\end{align*}
From this, Lemma \ref{emb}, and \eqref{eq-pr-pn},
we deduce that, for any $(x,t)\in\rp$,
\begin{align*}
&|f_k(x,t+\epsilon_k)-f(x,t+\epsilon_k)|\\
&\quad=[\mu(B(x,t))]^{-1/s}\left[\int_{B(x,t)}
|f_k(x,t+\epsilon_k)-f(x,t+\epsilon_k)|^sw(y)\,dy\right]^{1/s}\\
&\quad\leq [\mu(B(x,t))]^{-1/s}\left\{\int_{B(x,t)}
[(f_k-f)^\ast(y)]^sw(y)\,dy\right\}^{1/s}\\
&\quad\leq [\mu(B(x,t))]^{-1/s}\|(f_k-f)^\ast\|_{L_w^s(\rn)}\lesssim[\mu(B(x,t))]^{-1/s}\|(f_k-f)^\ast\|_{X}\\
&\quad\sim[\mu(B(x,t))]^{-1/s}\|f_k-f\|_{H_X(\rp)}\to 0
\end{align*}
as $k\to\infty$, where $w$ and $s$ are the same as in Lemma \ref{emb}.
This implies that \eqref{3.1x} holds
true and hence finishes the proof of Lemma \ref{p-n}.
\end{proof}

Next, we show Proposition \ref{lem-hx-hx2}.

\begin{proof}[Proof of Proposition \ref{lem-hx-hx2}]
We first prove the necessity.
To this end, fix a $u\in H_{X,2}(\rp)$. By Definition \ref{def-h-x}(iii),
we find a sequence $\{u_k\}_{k\in\nn}\subset H_X(\rp)\cap H^2(\rp)$ such
that
\begin{equation}\label{eq-uk-u}\lim_{k\to\infty} u_k=u
\end{equation}
in $H_X(\rp)$.
For any $k\in\nn$ and $\epsilon\in(0,\infty)$,
let $f_{k,\epsilon}(\cdot):=u_k(\cdot,\epsilon)$.
Since $u_k\in H^2(\rp)$, from \cite[p.\,51, Lemmas 2.6 and 2.7]{sw71},
it follows that, for any $(x,t)\in\rp$,
\begin{equation}\label{eq-u-fp}
P_t\ast f_{k,\epsilon}(x)=u_k(x,t+\epsilon).
\end{equation}
Using this, \eqref{eq-max-p}, and Theorem \ref{thm-p}, we conclude that
\begin{align}\label{eq-b}
&\sup_{\epsilon\in(0,\infty)}\|f_{k,\epsilon}\|_{H_X(\rn)}\\
&\quad\sim \sup_{\epsilon\in(0,\infty)}\|M(f_{k,\epsilon};P)\|_{X}
\sim \sup_{\epsilon\in(0,\infty)}\left\|\sup_{\{(y,t)\in\rp:\ |y-\cdot|<t\}}|P_t\ast f_{k,\epsilon}(y)|\right\|_{X}\notag\\
&\quad\sim\sup_{\epsilon\in(0,\infty)}\left\|\sup_{\{(y,t)\in\rp:\ |y-\cdot|<t\}}
|u_k(y,t+\epsilon)|\right\|_{X}\lesssim \sup_{\epsilon\in(0,\infty)}
\left\|u_k^\ast\right\|_{X}\notag\\
&\quad\sim \|u_k\|_{H_X(\rp)}\lesssim \|u\|_{H_X(\rp)},\notag
\end{align}
where $M(f_{k,\epsilon};P)$ is the same as in \eqref{eq-max-p} via $f$
replaced by $f_{k,\epsilon}$.
Thus, for any $k\in\nn$,
$\{f_{k,\epsilon}\}_{\epsilon\in(0,\infty)}$
is a bounded set in $H_X(\rn)$.
Moreover, by \cite[Lemma 4.8.18]{LYH2022}, we
find that $H_X(\rn)$ continuously embeds into
$\cs'(\rn)$, which further implies
that $\{f_{k,\epsilon}\}_{\epsilon\in(0,\infty)}$
is a bounded set in $\cs'(\rn)$.
By the weak compactness of $\cs'(\rn)$ (see, for
instance, \cite[p.\,119]{s93}), we find an
$f_k\in \cs'(\rn)$ and a sequence $\{\epsilon_j\}_{j\in\nn}$,
satisfying $\lim_{j\to\infty} \epsilon_j=0$, such that
$\lim_{j\to\infty}f_{k,\epsilon_j}=f_k$
in $\cs'(\rn)$, which implies that,
for any $\Phi\in\cs(\rn)$ satisfying $\int_{\rn}\Phi(y)\,dy\neq 0$ and
for any $(x,t)\in\rp$,
$$
\lim_{j\to\infty}[f_{k,\epsilon_j}\ast \Phi_t (x)]
=f_k\ast\Phi_t(x)
$$
and hence, for any $x\in\rn$,
\begin{align*}
M(f_k;\Phi)(x)
&=\sup_{t\in(0,\infty)}\lim_{j\to\infty}[f_{k,\epsilon_j}\ast \Phi_t (x)]\leq\sup_{t\in(0,\infty)}\varliminf_{j\to\infty}\sup_{t\in(0,\infty)}[f_{k,\epsilon_j}\ast \Phi_t (x)]\\
&\leq \varliminf_{j\to\infty}M(f_{k,\epsilon_j};\Phi)(x),
\end{align*}
where $M(f_{k,\epsilon_j};\Phi)$ is the same as in \eqref{2.6x}
via replacing $f$ by $f_{k,\epsilon_j}$.
From this, \cite[Theorem 3.1]{shyy}, Lemma \ref{fatou},
and \eqref{eq-b},
we deduce that, for any $k\in\nn$,
\begin{align}\label{eq-fu}
\|f_k\|_{H_X(\rn)}&\sim\|M(f_k;\Phi)\|_X \lesssim \varliminf_{j\to\infty}
\|M(f_{k,\epsilon_j};\Phi)\|_X\\
&\sim \varliminf_{j\to\infty}
\|f_{k,\epsilon_j}\|_{H_X(\rn)}\lesssim \|u\|_{H_X(\rp)},\notag
\end{align}
which further implies that $\{f_k\}_{k\in\nn}$
is also a bounded sequence in $H_X(\rn)$ and hence in $\cs'(\rn)$.
By the weak compactness of $\cs'(\rn)$ again and the diagonal principle,
we find an $f\in \cs'(\rn)$ and a subsequence
$\{f_{k_j,\epsilon_j}\}_{j\in\nn}$ of $\{f_{k,\epsilon_j}\}_{j\in\nn}$ such that
\begin{equation}\label{eq-fk-f}
\lim_{j\to\infty} k_j=\infty,\
\lim_{j\to\infty}\epsilon_j=0,
\ \text{and}\
\lim_{j\to\infty}f_{k_j,\epsilon_j}=f
\end{equation}
in $\cs'(\rn)$.
Applying an argument similar to that used in the estimation of \eqref{eq-fu},
we conclude that $\|f\|_{H_X(\rn)}\lesssim\|u\|_{H_X(\rp)}$
and hence $f\in H_X(\rn)$.
Form \eqref{eq-uk-u}, Lemma \ref{p-n}, \eqref{eq-u-fp}, \eqref{eq-fk-f},
Lemma\ref{d-n}, and Proposition \ref{well-def},
we deduce that, for any $(x,t)\in\rp$,
\begin{align*}
u(x,t)=\lim_{j\to\infty} u_{k_j}(x,t+\epsilon_j)
=\lim_{j\to\infty} P_t\ast f_{k_j,\epsilon_j}(x)
=P_t\ast f(x).
\end{align*}
This finishes the proof of the necessity.

Now, we show the sufficiency. To this end, assume that
$f\in H_X(\rn)$ such that, for any $(x,t)\in\rp$,
$u(x,t)=P_t\ast f(x)$.
Then, by Corollary \ref{co-pt}(i),
we know that $u$ is harmonic on $\rp$.
Also, by both Lemma \ref{den-h} and \eqref{def-pi}, we find a
sequence $\{f_k\}_{k\in\nn}\subset H_X(\rn)\cap L^2(\rn)$
such that, for any $(x,t)\in\rp$,
\begin{equation}\label{eq-u-f}
u(x,t)=\lim_{k\to\infty}P_t\ast f_k(x).
\end{equation}
For any $k\in\nn$ and $(x,t)\in\rp$,
let $u_k(x,t):=P_t\ast f_k(x)$. From \cite[Theorem 2.1(a)]{sw71},
we deduce that $\{u_k\}_{k\in\nn}\subset H^2(\rp)$.
Moreover, by Theorem \ref{thm-p}, we find that, for any $k\in\nn$,
\begin{equation*}
\|u_k\|_{H_X(\rp)}=\left\|\sup_{\{(y,t)\in\rp:\ |y-\cdot|<t\}}
|P_t\ast f_k(y)|\right\|_X=\left\|
M(f_k;P)\right\|_X\lesssim\|f_k\|_{H_X(\rn)},\end{equation*}
which further implies that
$\{u_k\}_{k\in\nn}\subset H_X(\rp)\cap H^2(\rp)$.
Finally, we show that $u\in H_X(\rp)$ and
$\lim_{k\to\infty}u_k=u$ in $H_X(\rp)$.
Indeed, observe that $\{f_k\}_{k\in\nn}$ is a Cauchy sequence
in $H_X(\rn)$. Thus, for any $\epsilon\in(0,\infty)$,
we find a $K\in\nn$ such that, for any
$k>K$ and $j>K$,
$$
\|f_k-f_j\|_{H_X(\rn)}<\epsilon.
$$
From this, Definition \ref{def-h-x}(ii),
\eqref{eq-u-f}, Lemma \ref{fatou},
and Theorem \ref{thm-p}, we deduce that,
for any $k>K$,
\begin{align*}
\|u_k-u\|_{H_X(\rp)}
&=\left\|\sup_{\{(y,t)\in\rp:\ |y-\cdot|<t\}}|u_k(y,t)-u(y,t)|\right\|_X\\
&=\left\|\sup_{\{(y,t)\in\rp:\ |y-\cdot|<t\}}
\lim_{j\to\infty}|P_t\ast(f_k-f_j)(y)|\right\|_X\\
&\leq \left\|\varliminf_{j\to\infty}M(f_k-f_j;P)\right\|_X\leq \varliminf_{j\to\infty}\left\|M(f_k-f_j;P)\right\|_X\\
&\lesssim \varliminf_{j\to\infty}\left\|f_k-f_j\right\|_{H_X(\rn)}\lesssim \epsilon,
\end{align*}
which implies that $u\in H_X(\rp)$ and
$\lim_{k\to\infty}u_k=u$ in $H_X(\rp)$.
Thus, $u\in H_{X,2}(\rp)$.
This finishes the proof of the sufficiency and hence Proposition \ref{lem-hx-hx2}.
\end{proof}

Next, we introduce the Hardy space $\bh_X(\rp)$
consisting of vectors of harmonic functions.
To be precise, let $F:= (u_0, u_1,\dots,u_n)$ be a harmonic
vector on $\rp$, namely, for any $k\in\{0,1,\dots,n\}$,
$u_k$ is harmonic on $\rp$.
The vector
$F$ is said to satisfy the \emph{generalized Cauchy--Riemann equation} if
\begin{align}\label{eq-cr}
\left\{
\begin{array}{lc}
\displaystyle{\sum_{j=0}^n\frac{\partial u_j}{\partial x_j}=0}, &\ \\
\displaystyle{\frac{\partial u_j}{\partial x_k}=\frac{\partial u_k}{\partial x_j}},
&\forall\, j,k\in\{0,1,\dots,n\},\\
\end{array}
\right.
\end{align}
where, for any $(x,t)\in\rp$,
$x:=(x_1,\dots,x_n)$ and $x_0:=t$.

\begin{definition}\label{hx-v}
Let $X$ be a BQBF space.
\begin{enumerate}
\item[{\rm(i)}]The \emph{Hardy space $\bh_X(\rp)$
consisting of vectors of harmonic functions} is defined
to be the set of all the harmonic vectors $F:= (u_0, u_1,\dots,u_n)$
on $\rp$ satisfying \eqref{eq-cr} and
$$
\|F\|_{\bh_X(\rp)}:=\sup_{t\in(0,\infty)}\|\,|F(\cdot,t)|\,\|_X <\infty,
$$
where, for any $(x,t)\in\rp$,
$$
|F(x,t)|:=\left[\sum_{j=0}^n|u_j(x,t)|^2\right]^{1/2}.
$$
When $X:=L^p(\rn)$ with $p\in(0,\infty)$, we denote $\bh_{L^p(\rn)}(\rp)$
simply by $\bh^p(\rp)$.
\item[{\rm(ii)}] The subspace $\bh_{X,2}(\rp)$ is defined to
be the set of all the vectors $F\in \bh_X(\rp)$ satisfying
that there exists a sequence $\{F_k\}_{k\in\nn}\subset
\bh_X(\rp)\cap \bh^2(\rp)$ such that
$$
F=\lim_{k\to\infty} F_k
$$
in $\bh_X(\rp)$.  Moreover, for any $F\in \bh_{X,2}(\rp)$,
let $\|F\|_{\bh_{X,2}(\rp)}:=\|F\|_{\bh_{X}(\rp)}$.
\end{enumerate}
\end{definition}

For any $F\in \bh_X(\rp)$, we have the following conclusion.

\begin{lemma}\label{f-fp}
Let $X$ be a BQBF space. Assume that
there exists an $s\in(\frac{n-1}{n}, \infty)$ such that $X^{1/s}$ is
a ball Banach function space and $M$ is bounded on $(X^{1/s})'$ .
Then, for any $a\in [0,\infty)$,
$q\in[\frac{n-1}{n},s)$, $F\in \bh_X(\rp)$, and $(x,t)\in\rp$,
\begin{equation}\label{eq-f-fp}
|F(x,t+a)|^q\leq  P_t\ast (g_a)^q(x),
\end{equation}
where $g_a(\cdot):=\lim_{t\to0^+}|F(\cdot,t+a)|$.
Moreover, for any $a\in[0,\infty)$, $g_a\in X$ and
$\|g_a\|_X\leq \|F\|_{\bh_X(\rp)}$.
\end{lemma}

To show Lemma \ref{f-fp}, let us first recall the concept of subharmonic functions
(see, for instance, \cite[p.\,76, (4.2)]{sw71}) .

\begin{definition}\label{def-sub-h}
Let $u$ be a continuous function on $\rp$. Then $u$
is said to be \emph{subharmonic on $\rp$} if, for any $\xi\in\rp$
and $r\in(0,\infty)$ such that
$\overline{B(\xi,r)}\subset\rp$,
$$
u(\xi)\leq \frac{1}{\omega_{n}}\int_{S^{n}} u(\xi+ts)\,ds,
$$
where $\omega_{n}$ is the \emph{unit spherical measure}
of $\mathbb{R}^{n+1}$
and $S^{n}$ is the \emph{unit sphere} of $\mathbb{R}^{n+1}$.
\end{definition}

To show Lemma \ref{f-fp}, we also need the following so-called
$B_{p}(\rn)$-condition of Muckenhoupt weights
(see, for instance, \cite[(2.21)]{ccyy16}).

\begin{lemma}\label{bpc}
Let $p\in[1,\infty)$ and $w\in A_p(\rn)$. Then there
exists a positive constant $C$ such that, for
any $(x,t)\in\rp$,
$$
\int_{\rn}\frac{w(y)}{(t+|x-y|)^{np}}\,dy
\leq C t^{-np}\int_{B(x,t)}w(y)\,dy.
$$
\end{lemma}

Now, we show Lemma \ref{f-fp}.

\begin{proof}[Proof of Lemma \ref{f-fp}]
For any $a\in[0,\infty)$ and
$t\in(0,\infty)$, let
$$
F_a(\cdot,t):=F(\cdot,t+a)
$$
and
$$
K(|F_a|^{q\eta},t):=\int_{\rn}\frac{|F(x,t+a)|^{q\eta}}{(|x|+1+t)^{n+1}}\,dx,
$$
where $\eta\in(1,s/q)$ (the existence of such an $\eta$
is guaranteed by $q\in[\frac{n-1}{n},s)$).
We claim that $K(|F_a|^{q\eta},\cdot)$ is bounded on $(0,\infty)$.
To show this, for any $t\in(0,\infty)$,
let
$$
E_t:=\{x\in\rn:\, |F(x,t+a)|\leq1\}.
$$
Then we write
\begin{align*}
K(|F_a|^{q\eta},t)
=\int_{E_t}\frac{|F(x,t+a)|^{q\eta}}{(|x|+1+t)^{n+1}}\,dx
+\int_{E_t^\complement}\cdots
&=: {\rm I} +{\rm II}.
\end{align*}
We first estimate ${\rm I}$. By the definition  of $E_t$,
we conclude that
\begin{align*}
{\rm I}&\leq \int_{\rn}\frac{1}{(|x|+1+t)^{n+1}}\,dx\lesssim 
\int_0^\infty\frac{1}{(r+1+t)^{n+1}}r^{n-1}\,dr\sim \frac{1}{1+t}.
\end{align*}
Next, we estimate ${\rm II}$. Letting $q_0:=s/(q\eta)$
and $w$ be the same as in Lemma \ref{emb},
then we know that $q_0\in(1,\infty)$.
By the H\"older inequality, we find that
\begin{align*}
{\rm II}&\leq \left\{\int_{\rn}|F(x,t+a)|^sw(x)\,dx\right\}^{1/q_0}
\left\{\int_{\rn}\frac{[w(x)]^{-q_0'/q_0}}{(|x|+
1+t)^{(n+1)q'_0}}\,dx\right\}^{1/q'_0}\\
&=: {\rm III}\times{\rm IV}.\notag
\end{align*}
From Lemma \ref{emb} and $F\in \bh_X(\rp)$, we deduce that
\begin{equation*}
{\rm III}\leq \|\,|F(\cdot,t+a)|\,\|_{L_w^s(\rn)}^{q\eta}\lesssim
\|\,|F(\cdot,t+a)|\,\|_X^{q\eta}\lesssim \|F\|_{\bh_X(\rp)}^{q\eta}.
\end{equation*}
By Remark \ref{def-w}, $q_0\in(1,\infty)$, and  \cite[Proposition 7.1.5(6)]{gl},
we conclude that $w\in A_{q_0}(\rn)$.
Moreover, from \cite[Proposition 7.1.5(4)]{gl}, we further
infer $w^{-q_0'/q_0}\in A_{q_0'}(\rn)$.
Using this and Lemma \ref{bpc} with $p:=q_0'$,
$w:=w^{-q_0'/q_0}$, and $(x,t):=(\mathbf{0},1)$, we find that
\begin{align*}
{\rm IV}&\leq\frac{1}{1+t}\left\{\int_{\rn}\frac{[w(x)]^{-q_0'/q_0}}{(|x|+
1)^{nq'_0}}\,dx\right\}^{1/q'_0}\\
&\lesssim\frac{1}{1+t}\left\{\int_{B(\mathbf{0},1)}
[w(x)]^{-q_0'/q_0}\,dx\right\}^{1/q'_0}
\sim \frac{1}{1+t}.
\end{align*}
Combining this and the estimates of ${\rm I}$, ${\rm II}$, and
${\rm III}$, we then finish the proof of the above claim.

On the other hand, by $F\in \bh_X(\rp)$ and \cite[p.\,234,
Theorem 4.14]{sw71}, we find that $|F_a|^q$ is subharmonic
for any $q\in[\frac{n-1}{n},s)$.
From \cite[Theorem 8]{n}, we deduce that
$|F_a|^q\leq  P_t\ast (g_a)^q$,
where, for any $x\in\rn$,
\begin{equation}\label{eq-g}
g_a(x):=\lim_{t\to 0^+}|F_a(x,t)|=\lim_{t\to 0^+}|F(x,t+a)|.
\end{equation}
This shows that \eqref{eq-f-fp} holds true.

Finally, by both Lemma \ref{fatou} and \eqref{eq-g}, we obtain, for any $a\in [0,\infty)$,
\begin{align*}
\|g_a\|_X=\left\|\lim_{t\to 0^+}|F(\cdot,t+a)|\right\|_X
\leq \varliminf_{t\to 0^+}\|\,|F(\cdot,t+a)|\,\|_X
\leq \|F\|_{\bh_X(\rp)}.
\end{align*}
This implies $g_a\in X$ and hence finishes the proof of Lemma \ref{f-fp}.
\end{proof}

Using Lemma \ref{f-fp}, we obtain the following conclusion.

\begin{proposition}\label{f-u}
Let $X$ be a BQBF space.
Assume that $X$ satisfies both Assumptions \ref{a} and
\ref{a2} with $s\in(\frac{n-1}{n},1]$ and $\theta\in[\frac{n-1}{n},s)$.
If $F:=(u_0,u_1,\dots,u_n)\in\bh_X(\rp)$, then
$u_0\in H_X(\rp)$ and there exists a positive constant $C$, independent of $F$, such that
\begin{equation*}
\|u_0\|_{H_X(\rp)}\leq C\|F\|_{\bh_X(\rp)}.
\end{equation*}
\end{proposition}

\begin{proof}
Assume that $F\in\bh_X(\rp)$. By Lemma \ref{f-fp},
we know that there exists a non-negative function
$h\in X$ such that, for any $(x,t)\in\rp$,
$$
|F(x,t)|^q\leq P_t\ast h^q(x),
$$
where $q\in[\frac{n-1}{n},\theta]$.
From this, we deduce that, for any $x\in\rn$,
\begin{align}\label{3.16x}
[u_0^\ast(x)]^q&=\sup_{\{(y,t)\in\rp:\ |y-x|<t\}}|u_0(y,t)|^q\leq \sup_{\{(y,t)\in\rp:\ |y-x|<t\}}|F(y,t)|^q\\
&\leq \sup_{\{(y,t)\in\rp:\ |y-x|<t\}}  P_t\ast h^q(y).\notag
\end{align}
Observe that, for any $x\in\rn$ and $(y,t)\in\rp$ with $|y-x|<t$,
\begin{align}\label{eq-hp-mh}
P_t\ast h^q(y)
&\sim \int_{\rn}\frac{t[h(z)]^q}{(t^2+|y-z|^2)^{(n+1)/2}}\,dz\\
&\sim \int_{B(x,2t)}\frac{t[h(z)]^q}{(t^2+|y-z|^2)^{(n+1)/2}}\,dz+\sum_{j=1}^\infty \int_{B(x,2^{j+1}t)\setminus B(x,2^jt)}
\cdots\notag\\
&\lesssim t^{-n}\int_{B(x,2t)}[h(z)]^q\,dz
+\sum_{j=1}^\infty 2^{-j}2^{-jn}t^{-n}
\int_{B(x,2^{j+1}t)}[h(z)]^q\,dz\notag\\
&\lesssim M(h^q)(x),\notag
\end{align}
which, together with \eqref{3.16x}, implies that
\begin{equation}\label{eq-u-m}
(u_0^\ast)^q\lesssim M(h^q).
\end{equation}
By Assumption \ref{a},
$q\in[\frac{n-1}{n},\theta]$, and \cite[Lemma 2.9]{cwyz},
we find that $M$ is bounded on $X^{1/q}$.
Combining this, \eqref{eq-u-m}, and Lemma \ref{f-fp},
we conclude that
\begin{align}\label{3.15x}
\|u_0\|_{H_X(\rp)}&=\|(u_0^\ast)^q\|_{X^{1/q}}^{1/q}
\lesssim\|M(h^q)\|_{X^{1/q}}^{1/q}\lesssim \|h^q\|_{X^{1/q}}^{1/q}\sim \|h\|_X\lesssim\|F\|_{\bh_X(\rp)}.
\end{align}
This finishes the proof of Proposition \ref{f-u}.
\end{proof}

Using Proposition \ref{f-u}, we obtain the following corollary.

\begin{corollary}\label{f-u-i}
Let $X$ be a BQBF space.
Assume that $X$ satisfies both Assumptions \ref{a} and
\ref{a2} with $s\in(\frac{n-1}{n},1]$ and $\theta\in[\frac{n-1}{n},s)$.
If $F:=(u_0,u_1,\dots,u_n)\in\bh_{X,2}(\rp)$, then
$u_0\in H_{X,2}(\rp)$ and there exists a positive constant $C$, independent of $F$, such that
\begin{equation*}
\|u_0\|_{H_X(\rp)}\leq C\|F\|_{\bh_X(\rp)}.
\end{equation*}
\end{corollary}

\begin{proof}
Fix an $F:=(u_0,u_1,\dots,u_n)\in\bh_{X,2}(\rp)$. By
Definition \ref{hx-v}(ii), for any $k\in\nn$, we  find an
\begin{equation}\label{eq-f-k}
F_k:=(u_{0,k},u_{1,k},\dots,u_{n,k})\in\bh_{X}(\rp)\cap\bh^{2}(\rp)
\end{equation}
such that
\begin{equation}\label{3.22x}
F=\lim_{k\to\infty}F_k
\end{equation}
in $\bh_{X}(\rp)$
and, for any $k\in\nn$,
\begin{equation}\label{eq-fk-fi}
\|F_k\|_{\bh_X(\rp)}\lesssim\|F\|_{\bh_X(\rp)}.
\end{equation}
From both Proposition \ref{f-u} and \eqref{eq-fk-fi}, we deduce that
$$u_0\in H_{X}(\rp)\ \text{and}\
\|u_0\|_{H_X(\rp)}\lesssim\|F\|_{\bh_X(\rp)}
$$
and, for any $k\in\nn$,
$$u_{0,k}\in H_{X}(\rp)\ \text{and}\
\left\|u_{0,k}\right\|_{H_X(\rp)}\lesssim\|F_k\|_{\bh_X(\rp)}\lesssim\|F\|_{\bh_X(\rp)}.
$$
On the other hand,  by both Definitions \ref{def-h-p} and
\ref{hx-v}, and by \eqref{eq-f-k}, we find  that, for any $k\in\nn$,
$$
\left\|u_{0,k}\right\|_{H^2(\rp)}\le
\sup_{t\in(0,\infty)}\||F_k(\cdot,t)|\|_{L^2(\rn)}=\|F_k\|_{\bh^2(\rp)}
$$
and hence
$u_{0,k}\in H_{X}(\rp)\cap H^2(\rp)$.
Moreover, using the definition of $\bh_X(\rp)$,
Proposition \ref{f-u}, and \eqref{3.22x},
we conclude that
$$
\left\|u_0-u_{0,k}\right\|_{H_X(\rp)}\lesssim\|F-F_k\|_{\bh_X(\rp)}\to 0
$$
as $k\to\infty$. This implies that $u_0\in H_{X,2}(\rp)$
and hence finishes the proof of Corollary \ref{f-u-i}.
\end{proof}

Now, we establish the relation between $H_X(\rn)$ and $\bh_{X,2}(\rp)$.

\begin{proposition}\label{u-f}
Assume that $X$ is a BQBF space, satisfies both Assumptions \ref{a} and
\ref{a2}, and has an absolutely continuous quasi-norm.
If $f\in H_X(\rn)$, then there exists an
$$
F:=(u_0,u_1,\dots,u_n)\in \bh_{X,2}(\rp)
$$
such that $F$ satisfies \eqref{eq-cr} and, for any
$(x,t)\in\rp$, $u_0(x,t)=P_t\ast f(x)$.
Moreover, there exists a positive constant $C$, independent of $f$,
such that $\|F\|_{\bh_X(\rp)}\leq C\|f\|_{H_X(\rn)}$.
\end{proposition}

To show Proposition \ref{u-f}, we need the following
boundedness of Riesz transforms on $H_X(\rn)$ which is a
combination of both \cite[Theorem 3.14]{wyy20}  and Lemma \ref{emb};
we omit the details here.

\begin{lemma}\label{bound-re}
Assume that $X$ is a BQBF space and satisfies both Assumptions \ref{a} and
\ref{a2}. Then, for any $j\in\{1,\dots,n\}$,
the Riesz transform $R_j$ is bounded on $H_X(\rn)$.
\end{lemma}

Next, we show Proposition \ref{u-f}.

\begin{proof}[Proof of Proposition \ref{u-f}]
 Assume $f\in H_X(\rn)$.
By Lemma \ref{den-h}, we find a sequence
$\{f_k\}_{k\in\nn}\subset H_X(\rn)\cap L^2(\rn)$
such that
\begin{equation}\label{eq-fk-f2}
\lim_{k\to\infty} f_k=f
\end{equation}
in both $H_X(\rn)$ and $\cs'(\rn)$ and, for any $k\in\nn$,
\begin{equation}\label{3.19x}
\left\|f_k\right\|_{H_X(\rn)}\lesssim \|f\|_{H_X(\rn)}.
\end{equation}
For any $k\in\nn$ and $(x,t)\in\rp$, let
$$
u_{0,k}(x,t) :=P_t\ast f_k(x)
$$
and, for any $j\in\{1,\dots,n\}$,
$$
u_{j,k}(x,t) :=Q_t^{(j)}\ast f_k(x),
$$
where, for any $j\in\{1,\dots,n\}$, $Q_t^{(j)}$ is the \emph{$j$-th
conjugate Poisson kernel} defined by setting, for any $(x,t)\in\rp$,
$$
Q_t^{(j)}(x):=c_{(n)}\frac{x_j}{(t^2+|x|^2)^{(n+1)/2}}
$$
with $c_{(n)}$ the same as in \eqref{unit-b}.
Moreover, for any $k\in\nn$, let
$$
F_k:=\left(u_{0,k},u_{1,k},\dots,u_{n,k}\right).
$$
From $f_k\in L^2(\rn)$ and \cite[p.\,236, Theorem 4.17(i)]{sw71},
we infer that, for any $k\in\nn$, $F_k\in\bh^2(\rp)$ and
satisfies \eqref{eq-cr}. Using Theorem \ref{thm-p} and
\eqref{3.19x}, we conclude that, for any $k\in\nn$
and $t\in(0,\infty)$,
\begin{equation}\label{3.19y}
\left\|u_{0,k}(\cdot,t)\right\|_X\lesssim\left\|f_k\right\|_{H_X(\rn)}\lesssim\|f\|_{H_X(\rn)}.
\end{equation}
Moreover, from \cite[p.\,65, Theorem 3 and p.\,78, Item 4.3]{s70},
we deduce that, for any $k\in\nn$, $j\in\{1,\dots,n\}$,
and $(x,t)\in\rp$,
$$
Q_t^{(j)}\ast f_k(x)=P_t\ast R_j(f_k)(x).
$$
This, together with Theorem \ref{thm-p},
Lemma \ref{bound-re}, and \eqref{3.19x}, implies that, for any
$k\in\nn$, $j\in\{1,\dots,n\}$,
and $t\in(0,\infty)$,
\begin{align*}
\left\|u_{j,k}(\cdot,t)\right\|_X
&=\left\|P_t\ast R_j(f_k)\right\|_X
\lesssim \left\|R_j(f_k)\right\|_{H_X(\rn)}\\
&\lesssim \left\|f_k\right\|_{H_X(\rn)}\lesssim \|f\|_{H_X(\rn)},
\end{align*}
which, combined with \eqref{3.19y},  shows that
\begin{equation}\label{eq-fk-f-norm}
\left\|F_k\right\|_{\bh_X(\rp)}=\sup_{t\in(0,\infty)}\left\||F_k(\cdot,t)|\right\|_X
\lesssim\sup_{t\in(0,\infty)}\left\|\sum_{j=0}^n
\left|u_{j,k}(\cdot,t)\right|\right\|_X\lesssim\|f\|_{H_X(\rn)}.
\end{equation}
Thus, for any $k\in\nn$,
\begin{equation}\label{f-h-h}
F_k\in \left[\bh_X(\rp)\cap\bh^2(\rp)\right].
\end{equation}

On the other hand, by \eqref{eq-fk-f2}, the well-known
boundedness of Riesz transforms on $L^2(\rn)$, and Lemma \ref{bound-re},
we find that, for any $j\in\{1,\dots,n\}$,
\begin{equation*}
\{R_j(f_k)\}_{k\in\nn}\subset H_X(\rn)\cap L^2(\rn)\
\text{and}\
\lim_{k\to\infty} R_j(f_k)=R_j(f)
\end{equation*}
in $H_X(\rn)$.
For any $(x,t)\in\rp$, let
$$
F(x,t):=(P_t\ast f(x), P_t\ast R_1(f)(x), \dots, P_t\ast R_n(f)(x)).
$$
Using Proposition \ref{well-def}, we find that, for any $j\in\{1,\dots,n\}$
and $(x,t)\in\rp$,
$$
\lim_{k\to\infty} P_t\ast R_j(f_k)(x)=P_t\ast R_j(f)(x)
$$
uniformly on any compact set of $\rp$. From this and the fact that
$F_k$ satisfies \eqref{eq-cr}, we deduce that $F$
also satisfies \eqref{eq-cr} and, for any $(x,t)\in\rp$,
$$
|F(x,t)|= \lim_{k\to\infty} \left|F_k(x,t)\right|.
$$
By this, Lemma \ref{fatou}, and \eqref{eq-fk-f-norm}, we conclude that
$$
\|F\|_{\bh_X(\rp)}=\sup_{t\in(0,\infty)}\left\|\lim_{k\to\infty}
\left|F_k(x,t)\right|\right\|_X
\leq \sup_{t\in(0,\infty)}\varliminf_{k\to\infty} \left\|\,
\left|F_k(x,t)\right|\,\right\|_X\lesssim \|f\|_{H_X(\rn)}.
$$
Thus, $F\in \bh_X(\rp)$. Finally, we show that
$F\in \bh_{X,2}(\rp)$.
Indeed, using the definitions of both $F$ and $F_k$, both (ii) and (iii) of
Corollary \ref{co-pt}, and Lemma \ref{bound-re}, we find that
\begin{align*}
\left\|F-F_k\right\|_{\bh_X(\rp)}
&\lesssim \sup_{t\in(0,\infty)}\left\|\left|P_t\ast f-P_t\ast f_k\right|
+\sum_{j=1}^n\left|P_t\ast R_j(f_k)-P_t\ast R_j(f)\right|\right\|_X\\
&\lesssim \sup_{t\in(0,\infty)}\left\|P_t\ast(f-f_k)\right\|_X+
\sup_{t\in(0,\infty)}\sum_{j=1}^n\left\|P_t\ast R_j(f-f_k)\right\|_X\\
&\lesssim \left\|f-f_k\right\|_{H_X(\rn)} \to 0
\end{align*}
as $k\to\infty$. This, together with \eqref{f-h-h},  proves
that $F\in \bh_{X,2}(\rp)$
and hence finishes the proof of Proposition \ref{u-f}.
\end{proof}

Combining Proposition \ref{lem-hx-hx2},
Corollary \ref{f-u-i},
and Proposition \ref{u-f},  we obtain the following
isomorphism theorem on $H_X(\rn)$, $H_{X,2}(\rp)$,
and $\bh_{X,2}(\rp)$; we omit the details here.

\begin{theorem}\label{isom}
Assume that $X$ is a BQBF space, satisfies both Assumptions \ref{a} and
\ref{a2} with $s\in(\frac{n-1}{n},1]$
and $\theta\in[\frac{n-1}{n},s)$, and has an absolutely
continuous quasi-norm.
Then the following statements are equivalent:
\begin{enumerate}
\item[{\rm (i)}] $u\in H_{X,2}(\rp)$;
\item[{\rm (ii)}] there exists an $f\in H_X(\rn)$ such
that $u(x,t)=P_t\ast f(x)$ for any $(x,t)\in\rp$;
\item[{\rm (iii)}] there exist harmonic functions $\{u_1,\dots,u_n\}$ on $\rp$
such that
$$F:=(u, u_1,\dots,u_n)\in \bh_{X,2}(\rp).$$
\end{enumerate}
Moreover,
$$
\|u\|_{H_{X,2}(\rp)}\sim \|f\|_{H_X(\rn)}\sim \|F\|_{\bh_{X,2}(\rp)},
$$
where the equivalence constants
are independent of $u$, $f$, and $F$.
\end{theorem}

Now, we establish the first order Riesz transform
characterization of $H_X(\rn)$. Let us first introduce
the following Riesz--Hardy space associated with $X$.

\begin{definition}\label{def-rhx}
Let $X$ be a BQBF space. The \emph{Riesz--Hardy space $H_{X,{\rm Riesz}}(\rn)$
associated with $X$} is defined to
be the set of all the $f\in\cs'(\rn)$
satisfying that there exists a positive constant $A$ and
a sequence $\{f_k\}_{k\in\nn}\subset L^2(\rn)$
such that
$\lim_{k\to\infty}f_k=f$
in $\cs'(\rn)$ and, for any $k\in\nn$,
\begin{equation}\label{eq-rh-norm}
\left\|f_k\right\|_X+\sum_{j=1}^n
\left\|R_j(f_k)\right\|_X\le A.
\end{equation}
Moreover, for any $f\in H_{X,{\rm Riesz}}(\rn)$,
$$
\|f\|_{H_{X,{\rm Riesz}}(\rn)}:=\inf\{A:\ A\ \text{satisfies \eqref{eq-rh-norm}}\}.
$$
\end{definition}

\begin{theorem}\label{thm-re}
Assume that $X$ is a BQBF space and satisfies both Assumptions \ref{a} and
\ref{a2} with $s\in(\frac{n-1}{n},1]$ and $\theta\in[\frac{n-1}{n},s)$.
Then the following statements hold true.
\begin{enumerate}
\item[{\rm(i)}] There exists a positive constant $C\in[0,\infty)$
such that, for any $f\in L^2(\rn)$,
$$
 \|f\|_{H_{X,{\rm Riesz}}(\rn)}\leq C \|f\|_{H_X(\rn)}.
$$
\item[{\rm(ii)}] There exists a positive constant $C\in[0,\infty)$
such that, for any $f\in \cs'(\rn)$,
$$
\|f\|_{H_X(\rn)}\leq C  \|f\|_{H_{X,{\rm Riesz}}(\rn)}.
$$
\item[{\rm(iii)}] Moreover, if $X$ has an absolutely continuous quasi-norm,
then
$$
H_X(\rn)=H_{X,{\rm Riesz}}(\rn)
$$
with equivalent (quasi-)norms.
\end{enumerate}
\end{theorem}

To prove Theorem \ref{thm-re}, we need the following
concept of the radial decreasing function.

\begin{definition}
A function $f$ on $\rn$ is said to be \emph{radial decreasing} if $f$ satisfies:
\begin{enumerate}
\item[{\rm(i)}] for any $x,y\in\rn$ with $|x|=|y|$, $f(x)=f(y)$;
\item[{\rm(ii)}] for any $t\in(0,\infty)$, let
$\widetilde{f}(t):=f(x)$, where $x\in\rn$ such that $|x|=t$.
Then $\widetilde{f}$ is decreasing on $(0,\infty)$.
\end{enumerate}
\end{definition}

Now, we prove Theorem \ref{thm-re}.

\begin{proof}[Proof of Theorem \ref{thm-re}]
We first show (i). To this end, fix an $f\in L^2(\rn)$.
If $\|f\|_{H_X(\rn)}=\infty$, then (i) obviously
holds true. In what follows, we assume that
$\|f\|_{H_X(\rn)}<\infty$.
Choose a $\Phi\in\cs(\rn)$ to be positive and  radial decreasing
such that $\int_{\rn}\Phi(x)\,dx=1$.
Since $f\in L^2(\rn)$, from \cite[Corollary 2.9]{d}, it follows that
\begin{equation}\label{eq-f-mf}
f\leq M\left(f;\Phi\right)\ \text{and}\
R_j(f)\leq M\left(R_j(f);\Phi\right),
\end{equation}
where $M(f;\Phi)$ is the same
as in \eqref{2.6x}  and  $M(R_j(f);\Phi)$ is the same as
in \eqref{2.6x} via $f$ replaced by $R_j(f)$.
Using this, Definition \ref{bfs}(ii),
\cite[Theorem 3.1(i)]{shyy}, and Lemma \ref{bound-re},
we conclude that
\begin{align}\label{f-re-f}
&\left\|f\right\|_X
+\sum_{j=1}^n\left\|R_j(f)\right\|_X\\
&\quad\leq\left\|M\left(f;\Phi\right)\right\|_X
+\sum_{j=1}^n\left\|M\left(R_j(f);\Phi\right)\right\|_X
\lesssim\|f\|_{H_X(\rn)},\notag
\end{align}
which implies that $f\in H_{X,{\rm Riesz}}(\rn)$ and
$\|f\|_{H_{X,{\rm Riesz}}(\rn)}\lesssim\|f\|_{H_X(\rn)}$.
This finishes the proof of (i).

We next show (ii). To achieve this, fix an $f\in \cs'(\rn)$.
If $\|f\|_{H_{X,{\rm Riesz}}(\rn)}=\infty$, then (ii) obviously
holds true.
In what follows, we assume that $\|f\|_{H_{X,{\rm Riesz}}(\rn)}<\infty$.
By Definition \ref{def-rhx},
we find a sequence  $\{f_k\}_{k\in\nn}\subset L^2(\rn)$
such that
\begin{equation}\label{3.24x}
\lim_{k\to\infty}f_k=f
\end{equation}
in $\cs'(\rn)$ and, for any $k\in\nn$,
\begin{equation}\label{fk-re-no}
\left\|f_k\right\|_X+\sum_{j=1}^n\left\|R_j(f_k)\right\|_X\lesssim \|f\|_{H_{X,{\rm Riesz}}(\rn)}.
\end{equation}
For any $k\in\nn$ and $(x,t)\in\rp$, let
\begin{equation}\label{eq-f-v}
F_k(x,t):=\left(P_t\ast f_k(x),P_t\ast R_1(f_k)(x), \dots, P_t\ast R_n(f_k)(x)\right).
\end{equation}
Using $f_k\in L^2(\rn)$ and \cite[p.\,236, Theorem 4.17(i)]{sw71},
we conclude that, for any $k\in\nn$, $F_k\in\bh^2(\rp)$.
From this and Lemma \ref{f-fp}, we deduce that,
for any $q\in[\frac{n-1}{n},\theta]$ and $(x,t)\in\rp$,
\begin{equation}\label{eq-fk-hkp}
\left|F_k(x,t)\right|^q\leq  P_t \ast \left[h_k\right]^q(x),
\end{equation}
where, for any $x\in\rn$, $h_k(x):=\lim_{t\to0^+}|F_k(x,t)|$.
By both $f_k\in L^2(\rn)$ and \cite[Corollary 2.9]{d},
we obtain, for almost every $x\in\rn$,
\begin{align}\label{eq-h}
h_k(x)&=\lim_{t\to0^+}\left[\left|P_t\ast f_k(x)\right|^2
+\sum_{j=1}^n\left|P_t\ast R_j(f_k)(x)\right|^2\right]^{1/2}\\
&=\left[\left|f_k(x)\right|^2
+\sum_{j=1}^n\left|R_j(f_k)(x)\right|^2\right]^{1/2}. \notag
\end{align}
Using \cite[Lemma 2.9]{cwyz}, 	Assumption \ref{a}, and
$q\in[\frac{n-1}{n},\theta]$, we find that $M$ is bounded on $X^{1/q}$.
This, combined with Definition \ref{bfs}(ii), \eqref{eq-fk-hkp},
\eqref{eq-hp-mh}, \eqref{eq-h}, and \eqref{fk-re-no},
implies that, for any $k\in\nn$,
\begin{align}\label{eq-fk-bh}
\left\|F_k\right\|_{\bh_X(\rp)}
&=\sup_{t\in(0,\infty)}\left\|\left|F_k(\cdot,t)\right|^q\right\|_{X^{1/q}}^{1/q}
\leq \sup_{t\in(0,\infty)}\left\|P_t
\ast \left[h_k\right]^q\right\|_{X^{1/q}}^{1/q}\\
&\lesssim \left\|M\left(\left[h_k\right]^q
\right)\right\|_{X^{1/q}}^{1/q}\lesssim \left\|h_k\right\|_X\notag\\
&\lesssim \|f_k\|_X+\sum_{j=1}^n\|R_j(f_k)\|_X
\lesssim \|f\|_{H_{X,{\rm Riesz}}(\rn)}.\notag
\end{align}
Thus, $\{F_k\}_{k\in\nn}\subset \bh_X(\rp)\cap\bh^2(\rp)$.
From this and Proposition \ref{f-u}, we deduce that,
for any $k\in\nn$, $P_t\ast f_k\in H_{X}(\rp)$ and
$$
\left\|P_t\ast f_k\right\|_{H_{X}(\rp)}\lesssim \left\|F_k\right\|_{\bh_X(\rp)}\lesssim \|f\|_{H_{X,{\rm Riesz}}(\rn)},
$$
which, together with Remark \ref{re-ptf}, Theorem \ref{thm-p}, \eqref{eq-max-p}, and both (i) and (ii) of Definition \ref{def-h-x},  further implies that,
for any $k\in\nn$, $f_k\in H_X(\rn)$
and
\begin{equation*}
\left\|f_k\right\|_{H_X(\rn)}\sim \|M(f_k;P)\|_{X}\sim \left\|P_t
\ast f_k\right\|_{H_{X}(\rp)}
\lesssim \|f\|_{H_{X,{\rm Riesz}}(\rn)}.
\end{equation*}
Using this, \cite[Theorem 3.1]{shyy}, \eqref{3.24x},
and Lemma \ref{fatou}, we obtain
\begin{align*}
\|f\|_{H_X(\rn)}&\lesssim \|M(f;\Phi)\|_X
\sim \left\|\sup_{t\in(0,\infty)}|f\ast\Phi_t|\right\|_X
\sim \left\|\sup_{t\in(0,\infty)}\left[\lim_{k\to\infty}\left|f_k
\ast\Phi_t\right|\right]\right\|_X\\
&\lesssim \varliminf_{k\to\infty}\left\|M\left(f_k;\Phi\right)\right\|_X
\lesssim \varliminf_{k\to\infty}\left\|f_k\right\|_{H_X(\rn)}
\lesssim \|f\|_{H_{X,{\rm Riesz}}(\rn)}.
\end{align*}
This finishes the proof of (ii).

Finally, we prove (iii). To achieve this, by (ii), it suffices to show that
\begin{equation}\label{eq-h-hr}
H_X(\rn)\subset H_{X,{\rm Riesz}}(\rn).
\end{equation}
To this end, fix an $f\in H_X(\rn)$.
By Lemma \ref{den-h},
we can find a sequence
$\{f_k\}_{k\in\nn}\subset H_X(\rn)\cap L^2(\rn)$
such that $\lim_{k\to\infty}f_k=f$
in both $H_X(\rn)$ and $\cs'(\rn)$ and
$\|f_k\|_{H_X(\rn)}\lesssim\|f\|_{H_X(\rn)}$.
Using this and \eqref{f-re-f},
we conclude that, for any $k\in\nn$,
\begin{align*}
\left\|f_k\right\|_X
+\sum_{j=1}^n\left\|R_j(f_k)\right\|_X
\lesssim \left\|f_k\right\|_{H_X(\rn)}\lesssim\|f\|_{H_X(\rn)},
\end{align*}
which implies that $f\in H_{X,{\rm Riesz}}(\rn)$ and
$\|f\|_{H_{X,{\rm Riesz}}(\rn)}\lesssim\|f\|_{H_X(\rn)}$.
This finishes the proof of \eqref{eq-h-hr} and hence Theorem \ref{thm-re}.
\end{proof}

\begin{remark}\label{re-sharp}
\begin{enumerate}
\item[{\rm(i)}] We point out that, when $X:=L^p(\rn)$
with $p\in(\frac{n-1}{n},1]$, Theorem \ref{thm-re} in this case was
obtained in \cite{fs} (see also \cite[p.\,123, Proposition 3]{s93}).
When $X:=L_w^1(\rn)$, Theorem \ref{thm-re} in this case was
obtained in \cite{w76}.
When $X$ is a variable Lebesgue space, Theorem \ref{thm-re} in this case was
obtained in \cite[Theorem 1.5]{yzn}.
When $X$ is a Musielak--Orlicz space, Theorem \ref{thm-re} in this case was
obtained in \cite[Theorem 1.5]{ccyy16}.
To the best of our knowledge, when $X$ is a Lorentz
space, a mixed-norm Lebesgue space,
a local generalized Herz space, or a mixed-norm Herz space,
both the results obtained in Theorem \ref{thm-re} are new.
When $X$ is a Morrey space, the results obtained
in (i) and (ii) of Theorem \ref{thm-re} are new.
Observe that the proof of Theorem \ref{thm-re}(iii) strongly
depends on the density of $H_X(\rn)\cap L^2(\rn)$
in $H_X(\rn)$. Therefore, it is unclear whether or not
Theorem \ref{thm-re}(iii) still holds true when $X$ is a Morrey space.
\item[{\rm(ii)}]
We also point out that the range of $\theta\in[\frac{n-1}{n},s)$
in Theorem~\ref{thm-re} is the best possible in the sense
that, for any $\theta\in(0,\frac{n-1}{n})$,
Theorem~\ref{thm-re} does not hold true anymore.
Indeed, let $X:=L^p(\rn)$ satisfy both Assumptions \ref{a} and
\ref{a2} with both $\theta\in(0,\frac{n-1}{n})$ and $s\in(\theta,1]$.
By \cite[Remark 2.4(a)]{wyy20}, we find that $p\in (0,\theta)
\subset (0,\frac{n-1}{n})$. However, it is known that, for any $p\in(0,\frac{n-1}{n}]$,
$H^p(\rn)$ can no longer be characterized by the
first order Riesz transforms but can be characterized by the higher order Riesz
transforms (see, for instance, \cite[p. 168]{fs} for more details).
This implies that Theorem~\ref{thm-re} does not hold true in this case.
Thus, the range of $\theta\in[\frac{n-1}{n},s)$
in Theorem~\ref{thm-re} is the best possible.
\item[{\rm(iii)}] It is worth pointing that, when $X:=L^p(\rn)$,
the range of index $p$ plays a vital role in considering
the first order Riesz transform characterization
of the Hardy space $H^p(\rn)$. That is because, for a harmonic function
$u$, only when $p\in (\frac{n-1}{n},\infty)$, $|u|^p$ is
subharmonic. This fact also results in that we have to find
an appropriate range of $q\in [\frac{n-1}{n},\infty)$ in
the proofs of Lemma \ref{f-fp}.
Notice that, when we establish the first order Riesz transform characterization
of $H_X(\rn)$ associated with a BQBF space $X$,
an essential difficulty is that the quasi-norm of the space $X$
has no explicit expression.
Moreover, a key tool used in the proof of Proposition \ref{f-u}, which
 strongly depends on Lemma \ref{f-fp},
is the boundedness of Hardy--Littlewood maximal
function on a convexification of $X$ [see \eqref{3.15x} above],
which follows from Assumption \ref{a}.
Notice that,
if $f_j\equiv 0$ for any $j\in\nn\cap[2,\infty)$ in \eqref{ma},
then \eqref{ma} becomes
$$
\|M(f)\|_{X^{1/\theta}}\lesssim \|f\|_{X^{1/\theta}},
$$
which implies that $s$ plays no role in this case.
Observe that, if $X:=L^p(\rn)$, then, in this case, $M$ is
bounded on $X^{1/\theta}$ if and only if  $p\in(\theta,\infty]$
and hence $\theta$ is the \emph{critical index} of $p$.
Thus, in some sense, $\theta$ can play the part of $p$
when $p$ is not available, namely, when the quasi-norm of $X$ under consideration
has no explicit expression. The conclusions of Theorem \ref{thm-re}
confirms this  observation.
\end{enumerate}
\end{remark}

\section{Higher Order Riesz Transform Characterization}\label{s4}

In this section, we establish the higher order Riesz
transform characterization of $H_X(\rn)$.
Let us begin with recalling the concept of the tensor
product of $m$ copies of $\mathbb{R}^{n+1}$.

\begin{definition}
Let $m\in\nn$ and $\{e_0, e_1, \dots, e_n\}$ be
an orthonormal basis of $\mathbb{R}^{n+1}$. The \emph{tensor
product of $m$ copies of $\mathbb{R}^{n+1}$} is defined to be the set
$$
\bigotimes^m\mathbb{R}^{n+1}:=\left\{\xi:=\sum_{j_1,\dots,j_m=0}^n
\xi_{j_1,\dots,j_m}e_{j_1}\otimes\cdots\otimes e_{j_m}:\
\left\{\xi_{j_1,\dots,j_m}\right\}_{j_1,\dots,j_m=0}^n\subset\mathbb{C}\right\},
$$
where $e_{j_1}\otimes\cdots\otimes e_{j_m}$ denotes the
\emph{tensor product} of $e_{j_1},\dots,e_{j_m}$
and
$$
\sum_{j_1,\dots,j_m=0}^n:=\sum_{j_1=0}^n\cdots\sum_{j_m=0}^n.
$$
Each $\xi\in\displaystyle{\bigotimes^m\mathbb{R}^{n+1}}$ is called a \emph{tensor of rank $m$}.
\end{definition}

Let $F:\ \rp\to \displaystyle{\bigotimes^m\mathbb{R}^{n+1}}$
be a tensor-valued function of rank $m$
on $\mathbb{R}^{n+1}$, that is, it has
the form that, for any $(x,t)\in\rp$,
\begin{equation}\label{def-tf}
F(x,t):=\sum_{j_1,\dots,j_m=0}^n
F_{j_1,\dots,j_m}(x,t)e_{j_1}\otimes\cdots\otimes e_{j_m},
\end{equation}
where, for any $j_1,\dots,j_m\in\{0,\dots,n\}$,
$F_{j_1,\dots,j_m}$ is a function from $\rp$ to $\mathbb{C}$.
A tensor-valued function $F$ of rank $m$ is said to be
\emph{symmetric} if, for any permutation $\sigma$ on
$\{1,\dots,m\}$, any $j_1, \dots, j_m\in\{0,\dots,n\}$, and any
$(x,t)\in\rp$,
$$
F_{j_1,\dots,j_m}(x,t)=F_{j_{\sigma(1)},\dots,j_{\sigma(m)}}(x,t).
$$
A symmetric tensor-valued function $F$ of rank $m$
is said to be of \emph{trace zero} if,
for any $(x,t)\in\rp$,
\begin{equation*}
\left\{
\begin{array}{lc}
\displaystyle{\sum_{j=0}^nF_{j,j, j_3\dots,j_m}(x,t)=0, \ \forall\,j_3, \dots, j_m\in\{0,\dots,n\},}\ &\text{if}\ m\geq 3,\\
\displaystyle{\sum_{j=0}^nF_{j,j}(x,t)=0}&\text{if}\ m=2.
\end{array}
\right.
\end{equation*}
We make the convention that a tensor-valued function $F$
of rank 1 is \emph{always of trace zero}.
Let $F$ be the same as in \eqref{def-tf} and, for any
$j_1, \dots, j_m\in\{0,\dots,n\}$,
$F_{j_1,\dots,j_m}$ be differentiable.
Then the \emph{gradient} of $F$,
$$
\nabla F:\ \rp\to \bigotimes^{m+1}\mathbb{R}^{n+1},
$$
is a tensor-valued function of rank $m+1$ of the form that,
for any $(x,t)\in\rp$,
\begin{align*}
\nabla F(x,t)&=\sum_{j=0}^n\frac{\partial F}{\partial x_j}(x,t)\otimes e_j\\
&=\sum_{j=0}^n\sum_{j_1,\dots,j_m=0}^n\frac{\partial F_{j_1,\dots,j_m}}{\partial x_j}(x,t) e_{j_1}\otimes\cdots\otimes e_{j_m}\otimes e_j.
\end{align*}
Here and thereafter, we always let $x_0:=t$.
A tensor-valued function $F$ is said to satisfy the \emph{generalized
Cauchy--Riemann equation} if both $F$ and $\nabla F$
are symmetric and of trace zero. Obviously, if $m = 1$,
this definition of the generalized Cauchy--Riemann equation is
equivalent to that   in \eqref{eq-cr}. For more details on
the generalized Cauchy--Riemann equation on tensor-valued
functions, we refer the reader to \cite{ps08,sw68}.

\begin{remark}\label{u-g-cr}
Let $u$ be a harmonic function on $\rp$ and, for any $m\in\nn$,
\begin{equation}\label{eq-def-nu}
\nabla^m u:=\{\partial^\alpha u\}_{\alpha
\in\zz_+^{n+1},\,|\alpha|=m},
\end{equation}
where,
for any $\alpha:=(\alpha_0,\dots, \alpha_n)\in\zz_+^{n+1}$,
$|\alpha|:=\sum_{j=0}^n\alpha_j$
and
$$
\partial^\alpha:=\left(\frac{\partial}{\partial x_0}\right)^{\alpha_0}
\cdots \left(\frac{\partial}{\partial x_n}\right)^{\alpha_n}.
$$
By \cite[Remark 2.13]{yzn}, we find that $\nabla^m u$
satisfies the generalized Cauchy--Riemann equation.
\end{remark}

Now, we establish the higher  order Riesz transform
characterization of $H_X(\rn)$. We first introduce
the concept of higher order Riesz--Hardy spaces.

\begin{definition}\label{def-h-rhx}
Let $m\in\nn$ and $X$ be a BQBF space. The \emph{m-th
Riesz--Hardy space $H_{X,{\rm Riesz}}^m(\rn)$
associated $X$} is defined to
be the set of all the $f\in\cs'(\rn)$
satisfying that there exists a positive constant $A$ and
a sequence $\{f_k\}_{k\in\nn}\subset L^2(\rn)$
such that
$$\lim_{k\to\infty}f_k=f$$
in $\cs'(\rn)$ and,
for any $k\in\nn$, $l\in\{1,\dots,m\}$,
and $j_1,\dots,j_l\in \{1,\dots,n\}$,
$$
f_k\in X,\ R_{j_1}\dots R_{j_l}(f_k)\in X,
$$
and
\begin{equation}\label{eq-h-rh-norm}
\left\|f_k\right\|_X+\sum_{l=1}^m
\sum_{j_1,\dots,j_l=1}^n\left\|R_{j_1}\cdots R_{j_l}(f_k)\right\|_X\leq A.
\end{equation}
Moreover, for any $f\in H_{X,{\rm Riesz}}^m(\rn)$,
$$
\|f\|_{H_{X,{\rm Riesz}}^m(\rn)}:=\inf\{A:\ A\ \text{satisfies \eqref{eq-h-rh-norm}}\}.
$$
\end{definition}

Obviously, when $m=1$, then $H_{X,{\rm Riesz}}^1(\rn)$
coincides with $H_{X,{\rm Riesz}}(\rn)$ in Definition \ref{def-rhx}.
The following theorem establishes the higher Riesz
transform characterization of $H_X(\rn)$.

\begin{theorem}\label{thm-h-re}
Let $m\in\nn\cap[2,\infty)$.
Assume that $X$ is a BQBF space and satisfies both Assumptions~\ref{a} and
\ref{a2} with $s\in(\frac{n-1}{n+m-1},1]$ and
$\theta\in[\frac{n-1}{n+m-1},s)$.
Then the following statements hold true.
\begin{enumerate}
\item[{\rm(i)}] There exists a positive constant $C\in[0,\infty)$
such that, for any $f\in L^2(\rn)$,
$$
 \|f\|_{H_{X,{\rm Riesz}}^m(\rn)}\leq C \|f\|_{H_X(\rn)}.
$$
\item[{\rm(ii)}] There exists a positive constant $C\in[0,\infty)$
such that, for any $f\in \cs'(\rn)$,
$$
\|f\|_{H_X(\rn)}\leq C  \|f\|_{H_{X,{\rm Riesz}}^m(\rn)}.
$$
\item[{\rm(iii)}] Moreover, if $X$ has an absolutely continuous quasi-norm,
then
$$
H_X(\rn)=H_{X,{\rm Riesz}}^m(\rn)
$$
with equivalent (quasi-)norms.
\end{enumerate}
\end{theorem}

To prove Theorem \ref{thm-h-re}, we need two lemmas.
The following lemma is just \cite[Theorem 1]{cz64}.

\begin{lemma}\label{lem-g-f}
Let $m\in\nn$ and $u$ be a harmonic function on $\rp$.
Then, for any $q\in[\frac{n-1}{n+m-1},\infty)$, $|\nabla^m u|^q$
is subharmonic, where $\nabla^m u$ is the same as in \eqref{eq-def-nu}
and
$$
|\nabla^m u|:=\left[\sum_{\alpha\in\zz_+^{n+1},
|\alpha|=m} |\partial^\alpha u|^2\right]^{1/2}.
$$
\end{lemma}

The following lemma is just
\cite[Theorem 14.3]{u01} (see also \cite{sw71}).

\begin{lemma}\label{lem-f-u}
Let $m\in\nn\cap[2,\infty)$, and $F$ be a tensor-valued function
of rank $m$ satisfying that both $F$ and $\nabla F$
are symmetric and that $F$ is of trace zero.
Then there exists a harmonic function $u$ on $\rp$
such that $\nabla^m u=F$, namely, for any
$j_1,\dots,j_m\in\{0,1,\dots,n\}$ and $(x,t)\in\rp$,
$$
\frac{\partial}{\partial x_{j_1}}\cdots \frac{\partial}{\partial x_{j_m}} u(x,t)
=F_{j_1,\dots, j_m}(x,t).
$$
\end{lemma}

Next, we introduce the concept of the Hardy space
$\bh_{X}^m(\rp)$ of tensor-valued functions
of rank $m$ associated with $X$.

\begin{definition}\label{def-hm-tv}
Let $X$ be a BQBF space. The \emph{Hardy space $\bh_{X}^m(\rp)$
of tensor-valued functions of rank $m$ associated
with $X$} is defined to be the
set of all the tensor-valued functions $F$ of rank $m$ on $\rp$
satisfying the generalized Cauchy--Riemann equation and
$$
\|F\|_{\bh_{X}^m(\rp)}:=\sup_{t\in(0,\infty)}\|\,|F(\cdot,t)|\,\|_X<\infty,
$$
where, for any $(x,t)\in\rp$,
$$
|F(x,t)|:=\left\{\sum_{j_1,\dots,j_m=0}^n
|F_{j_1,\dots,j_m}(x,t)|^2\right\}^{1/2}.
$$
\end{definition}

\begin{remark}\label{re-4.6x}
We point out that, in both Lemma \ref{f-fp} and Proposition \ref{f-u},
we used the restriction that $\theta \in [\frac{n-1}{n},s)$
only because, for any $q\in[\frac{n-1}{n},\infty)$,
the $q$-power of the absolute value of the first order
gradient, $|\nabla u|^q$, of a harmonic function $u$ on
$\rp$ is subharmonic. By Lemma \ref{lem-g-f}, we know that, for any
$m\in\nn$ and $q\in[\frac{n-1}{n+m-1},\infty)$,
$|\nabla^m u|^q$ is subharmonic on $\rp$. Therefore, the restriction
$\theta \in [\frac{n-1}{n},s)$ in both Lemma \ref{f-fp} and Proposition \ref{f-u}
can be relaxed to
$\theta \in [\frac{n-1}{n+m-1},s)$ when we deal with the Hardy
space $\bh_{X}^m(\rp)$
of tensor-valued functions of rank $m$
instead of  the Hardy
space $\bh_{X}(\rp)$ of harmonic vectors.
Then we can show that both Lemma \ref{f-fp} and Proposition\ref{f-u}
still hold true for any $q\in[\frac{n-1}{n+m-1},\theta]$; we omit the details.
Moreover, for any BQBF space $X$ satisfying both Assumptions
\ref{a} and \ref{a2} with $s\in(0,1]$ and $\theta\in(0,s)$,
we can always find a sufficiently large $m$ such that
$\theta \in [\frac{n-1}{n+m-1},s)$.
\end{remark}

Now, we show Theorem \ref{thm-h-re}.

\begin{proof}[Proof of Theorem \ref{thm-h-re}]
We first prove (i). To this end, fix an $f\in L^2(\rn)$.
If $\|f\|_{H_X(\rn)}=\infty$, then (i) obviously
holds true. In what follows, we assume that
$\|f\|_{H_X(\rn)}<\infty$.
Choose a $\Phi\in\cs(\rn)$ to be positive and  radial
decreasing such that $\int_{\rn}\Phi(x)\,dx=1$.
By this, \eqref{eq-f-mf}, Definition \ref{bfs}(ii), \cite[Theorem
3.1(i)]{shyy}, and Lemma \ref{bound-re}, we conclude that
\begin{align}\label{4.4x}
&\|f\|_X+\sum_{k=1}^m\sum_{j_1,\dots,j_k=1}^n\|R_{j_1}\dots R_{j_k}(f)\|_X\\
&\quad\lesssim \|M(f;\Phi)\|_X+\sum_{k=1}^m\sum_{j_1,\dots,j_k=1}^n\|M(R_{j_1}\dots R_{j_k}(f);\Phi)\|_X\notag\\
&\quad\lesssim \|f\|_{H_X(\rn)}+\sum_{k=1}^m\sum_{j_1,\dots,j_k=1}^n\|R_{j_1}\dots R_{j_k}(f)\|_{H_X(\rn)}\lesssim \|f\|_{H_X(\rn)}\notag.
\end{align}
This finishes the proof of (i).

Next, we show (ii). To achieve this, fix an $f\in\cs'(\rn)$.
If $\|f\|_{H_{X,{\rm Riesz}}^m(\rn)}=\infty$, then (ii) obviously
holds true.
In what follows, we assume that $\|f\|_{H_{X,{\rm Riesz}}^m(\rn)}<\infty$.
By Definition \ref{def-h-rhx}, we find
a sequence $\{f_k\}_{k\in\nn}\subset L^2(\rn)$  such that
\begin{equation}\label{eq-fk-f-i}
\lim_{k\to\infty}f_k=f
\end{equation}
in $\cs'(\rn)$ and, for any $k\in\nn$,
\begin{equation}\label{f-rr}
\left\|f_k\right\|_X+\sum_{k=1}^m
\sum_{j_1,\dots,j_k=1}^n\left\|R_{j_1}\cdots R_{j_k}
(f_k)\right\|_X\lesssim \|f\|_{H_{X,{\rm Riesz}}^m(\rn)}.
\end{equation}
For any $k\in\nn$, $j_1,\dots,j_m\in \{0,\dots,n\}$, and $(x,t)\in\rp$, write
$$
F_{j_1,\dots,j_m,k}(x,t):= P_t\ast (R_{j_1}\cdots R_{j_m}(f_k))(x),
$$
where $R_0:=I$ is the identity operator, and let
\begin{equation*}
F_k(x,t):=\sum_{j_1,\dots,j_m=0}^n
F_{j_1,\dots,j_m,k}(x,t)e_{j_1}\otimes\cdots\otimes e_{j_m}.
\end{equation*}
From this and the proof of \cite[Lemma 17.1]{u01}, we deduce
that, for any $k\in\nn$,
$F_k$ satisfies the generalized Cauchy--Riemann equation.
Using this and \cite[Lemma 17.2]{u01}, we have,
for any $k\in\nn$, $q\in[\frac{n-1}{n+m-1},\theta]$, and
$(x,t)\in\rp$,
\begin{equation}\label{eq-f-fp-t}
|F_k(x,t)|^q\leq  P_t \ast |h_k|^q(x),
\end{equation}
where, for any $x\in\rn$,
$$
h_k(x):=\{R_{j_1}\cdots R_{j_m}(f_k)(x)\}_{j_1,\dots,j_m\in \{0,\dots,n\}}
$$
and
$$
|h_k(x)|:=\left\{\sum_{j_1,\dots,j_m=0}^n
|R_{j_1}\cdots R_{j_m}(f_k)(x)|^2\right\}^{1/2}.
$$
By this, Definition \ref{bfs}(ii), \eqref{eq-f-fp-t},
and \eqref{f-rr}, we conclude  that, for any $k\in\nn$,
\begin{align*}
\left\|F_k\right\|_{\bh_{X}^m(\rp)}
&=\sup_{t\in(0,\infty)}\left\||F_k(\cdot,t)|^q\right\|_{X^{1/q}}^{1/q}
\leq \sup_{t\in(0,\infty)}\left\|P_t \ast |h_k|^q\right\|_{X^{1/q}}^{1/q}\\
&\lesssim \left\|M(|h_k|^q)\right\|_{X^{1/q}}^{1/q}\lesssim \left\|\,|h_k|\,\right\|_X\\
&\lesssim \|f_k\|_X+\sum_{j_1,\dots,j_m=0}^n\|R_{j_1}\dots R_{j_m}(f_k)\|_X\lesssim \|f\|_{H_{X,{\rm Riesz}}(\rn)},
\end{align*}
which implies that $F_k\in\bh_{X}^m(\rp)$.
Combining this and Remark \ref{re-4.6x} (on the higher order counterparts
of both Lemma \ref{f-fp} and Proposition \ref{f-u}),
we obtain, for any $k\in\nn$ and $t\in(0,\infty)$,
$$
\left\|P_t\ast f_k\right\|_{H_{X}(\rp)}\lesssim \left\|F_k\right\|_{\bh_X^m(\rp)}\lesssim \|f\|_{H_{X,{\rm Riesz}}(\rn)}.
$$
Applying this
and an argument similar to that used in the proof
of Theorem \ref{thm-re}(ii), we find that $f\in H_X(\rn)$
and $\|f\|_{H_X(\rn)}\lesssim \|f\|_{H_{X,{\rm Riesz}}(\rn)}$.
This finishes the proof of (ii).

Finally, we show (iii). To this end, by (ii), it suffices to prove that
\begin{equation}\label{eq-h-hr-high}
H_X(\rn)\subset H_{X,{\rm Riesz}}^m(\rn).
\end{equation}
To this end, fix an $f\in H_X(\rn)$.
By Lemma \ref{den-h},
we can find a sequence
$\{f_k\}_{k\in\nn}\subset H_X(\rn)\cap L^2(\rn)$
such that $\lim_{k\to\infty}f_k=f$
in both $H_X(\rn)$ and $\cs'(\rn)$ and
$\|f_k\|_{H_X(\rn)}\lesssim\|f\|_{H_X(\rn)}$.
From this and \eqref{4.4x},
we deduce that, for any $k\in\nn$,
\begin{align*}
\left\|f_k\right\|_X
+\sum_{k=1}^m\sum_{j_1,\dots,j_k=1}^n\|R_{j_1}\dots R_{j_k}(f_k)\|_{X}
\lesssim \left\|f_k\right\|_{H_X(\rn)}\lesssim\|f\|_{H_X(\rn)},
\end{align*}
which implies that $f\in H_{X,{\rm Riesz}}(\rn)$ and
$\|f\|_{H_{X,{\rm Riesz}}(\rn)}\lesssim\|f\|_{H_X(\rn)}$.
This finishes the proof of \eqref{eq-h-hr-high} and hence Theorem \ref{thm-h-re}.
\end{proof}

\begin{remark}
We point out that, when $X:=L^p(\rn)$
with $p\in(\frac{n-1}{n},1]$, Theorem \ref{thm-h-re} in this case was
obtained \cite{fs} (see also \cite[p.\,133, Item 5.16]{s93}).
When $X$ is a variable Lebesgue space, Theorem \ref{thm-h-re} in this case was
obtained in \cite[Theorem 1.6]{yzn}.
When $X$ is a Musielak--Orlicz space, Theorem \ref{thm-h-re} in this case was
obtained in \cite[Theorem 1.7]{ccyy16}.
To the best of our knowledge, when $X$ is a   Lorentz
space, a mixed-norm Lebesgue space,
a local generalized Herz space, or a mixed-norm Herz space,
the results obtained in Theorem \ref{thm-h-re} are new.
Observe that the proof of Theorem \ref{thm-h-re} strongly
depends on the density of $H_X(\rn)\cap L^2(\rn)$
in $H_X(\rn)$. Thus, it is unclear whether or not
Theorem \ref{thm-h-re} still holds true when $X$ is a Morrey space.
\end{remark}

\section{Applications}\label{s5}

In this section, we apply our main results, Theorems \ref{thm-re}
and \ref{thm-h-re},  to
five concrete examples of ball quasi-Banach
function spaces, namely, Lorentz spaces (Subsection \ref{ls}),
mixed-norm Lebesgue spaces (Subsection
\ref{mnls}), local
generalized Herz spaces (Subsection
\ref{lghs}), mixed-norm Herz spaces
(Subsection \ref{mhs}), and Morrey spaces (Subsection \ref{ms}),
and give the Riesz transform characterizations of  Hardy type
spaces based on these ball quasi-Banach
function spaces.
These examples indicate both the practicality
and the operability of the main results of this article and
more applications to new function spaces are obviously possible.

\subsection{Lorentz Spaces}\label{ls}

In this section, we apply our main results to the
Lorentz space.
Let us begin with the following concept of Lorentz spaces.

\begin{definition}
Let $p\in(0,\fz)$ and $r\in(0,\fz]$. The \emph{Lorentz space} $L^{p,r}(\rn)$ is defined to be
the set of all the measurable functions $f$ on $\rn$ satisfying that, when $p,r\in(0,\fz)$,
$$\|f\|_{L^{p,r}(\rn)}:=
\left\{\int_0^\infty[t^{1/p}f^\ast(t)]^r\,\frac{dt}{t}\right\}^{\frac{1}{r}}<\fz$$
and, when $p\in(0,\fz)$ and $r=\fz$,
$$\|f\|_{L^{p,q}(\rn)}:=\sup_{t\in(0,\fz)}\lf\{t^{1/p}f^\ast(t)\r\}<\fz,$$
where $f^\ast$ denotes the \emph{decreasing rearrangement function} of $f$, which is defined by setting,
for any $t\in[0,\fz)$,
$$f^\ast(t):=\inf\{s\in(0,\fz):\ \mu_f(s)\le t\}$$
with $\mu_f(s):=|\{x\in\rn:\ |f(x)|>s\}|$.
\end{definition}

Applying both Theorems \ref{thm-re} and \ref{thm-h-re},
we have the following  Riesz
transform characterization of Lorentz--Hardy spaces.

\begin{theorem}\label{riesz-ls}
Let $m\in\nn$, $p\in(\frac{n-1}{n+m-1},\infty)$,
$r\in(0,\fz)$,  and $X:=L^{p,r}(\rn)$.
Then $H_X(\rn)=H_{X,{\rm Riesz}}^m(\rn)$
with equivalent (quasi-)norms.
\end{theorem}

\begin{proof}
By \cite[Theorem 2.3(iii)]{ccmp}, we conclude that
$L^{p,r}(\rn)$ satisfies both Assumptions
\ref{a} and \ref{a2} with $s\in(\frac{n-1}{n+m-1},\min\{1,p\}]$,
$\theta\in[\frac{n-1}{n+m-1},s)$,
and $q\in(\max\{1,p,r\},\infty)$. Moreover, from \cite[Remark 3.4(iii)]{wyy20},
we infer that $L^{p,r}(\rn)$ has an absolutely continuous quasi-norm.
Then, using both Theorems \ref{thm-re}(iii) and \ref{thm-h-re}(iii) with $X:=L^{p,r}(\rn)$,
we obtain the desired conclusion,
which completes the proof of
Theorem \ref{riesz-ls}.
\end{proof}

\subsection{Mixed-Norm Lebesgue Spaces}\label{mnls}

The mixed-norm Lebesgue space
$L^{\vec{p}}\left(\mathbb{R}^{n}\right)$
was studied by Benedek and Panzone
\cite{BAP1961} in 1961, which can be
traced back to H\"ormander \cite{HL1960}.
Later on, in 1970, Lizorkin \cite{LPI1970}
further developed both the theory of
multipliers of Fourier integrals and
estimates of convolutions in the
mixed-norm Lebesgue spaces. Particularly,
in order to meet the requirements arising
in the study of the boundedness of
operators, partial differential equations,
and some other analysis subjects, the
real-variable theory of mixed-norm
function spaces has rapidly been
developed in recent years (see, for
instance,
\cite{CGN2017,CGG2017,CGN20172,HLY2019,HLYY2019,HYacc}).

\begin{definition}\label{mixed}
Let $\vec{p}:=(p_{1}, \ldots, p_{n})
\in(0, \infty]^{n}$. The \emph{mixed-norm
Lebesgue space
$L^{\vec{p}}\left(\mathbb{R}^{n}\right)$}
is defined to be the set of all the measurable
functions $f$ on $\rn$ such that
\begin{align*}
\|f\|_{L^{\vec{p}}\left(\mathbb{R}^{n}\right)}:=\left\{\int_{\mathbb{R}}
\cdots\left[\int_{\mathbb{R}}\left|f\left(x_{1}, \ldots,
x_{n}\right)\right|^{p_{1}} \,d x_{1}\right]^{\frac{p_{2}}{p_{1}}}
\cdots \,dx_{n}\right\}^{\frac{1}{p_{n}}}<\infty
\end{align*}
with the usual modifications made when $p_i=\infty$ for some $i\in\{1,\ldots,n\}$.
Moreover, let
\begin{equation}\label{p-p+}
p_-:=\min\{p_{1}, \ldots, p_{n}\}\quad  	
\text{and}\quad p_+:=\max\{p_{1}, \ldots, p_{n}\}.
\end{equation}

\end{definition}

\begin{remark}\label{mix-r}
Let $\vec{p}:=(p_{1}, \ldots, p_{n})\in(0, \infty)^{n}$ and both
$p_-$ and $p_+$ be the same as in \eqref{p-p+}.
By Definition \ref{mixed}, we easily conclude that
$L^{\vec{p}}\left(\mathbb{R}^{n}\right)$
is a ball quasi-Banach space,
but it is worth pointing out that $L^{\vec{p}}(\rn)$
may not be a quasi-Banach function space
(see, for instance, \cite[Remark 7.20]{zyyw}).
\end{remark}

On the mixed-norm Hardy space, we have the following Riesz
transform characterization.

\begin{theorem}\label{riesz-mnls}
Let $m\in\nn$, $\vec{p}:=(p_{1}, \ldots, p_{n})\in(0, \infty)^{n}$,
$p_-$ be the same as in \eqref{p-p+}, and
$X:=L^{\vec{p}}\left(\mathbb{R}^{n}\right)$.
If $p_-\in(\frac{n-1}{n+m-1},\infty)$,
then $H_X(\rn)=H_{X,{\rm Riesz}}^m(\rn)$
with equivalent (quasi-)norms.
\end{theorem}

\begin{proof}
Choose an $s\in(\frac{n-1}{n+m-1},\min\{p_-,1\}]$,
a $\theta\in[\frac{n-1}{n+m-1},s)$,
and a $q\in(\max\{p_+,1\},\infty)$,
where $p_+$ is the same as in \eqref{p-p+}.
Then, by \cite[Lemma
3.7]{HLY2019}, the dual theorem of
$L^{\vec{p}}(\rn)$ (see \cite[p.\,304,
Theorem 1.a]{BAP1961}), and
\cite[Lemma 3.5]{HLY2019}, we conclude that
$L^{\vec{p}}(\rn)$ satisfies both Assumptions
\ref{a} and \ref{a2} with $s$,
$\theta$,
and $q$ chosen as above. Moreover, from the dominated convergence theorem,
we infer that $L^{\vec{p}}(\rn)$ has an absolutely continuous quasi-norm.
Then, using Theorems~\ref{thm-re}(iii) and \ref{thm-h-re}(iii)
with $X:=L^{\vec{p}}\left(\mathbb{R}^{n}\right)$,
we obtain the desired conclusion,
which completes the proof of
Theorem \ref{riesz-mnls}.
\end{proof}

\subsection{Local
Generalized Herz Spaces}\label{lghs}

In this section, we apply our main results to
the local generalized Herz
space (see, for instance, \cite{rs20,LYH2022}). Let us begin
with the concepts of the almost decreasing function (see, for instance,
\cite[p.\,30]{kmrs}) and the function
class $M\left(\mathbb{R}_{+}\right)$ (see, for instance,
\cite[Definition 2.1]{rs20}).

\begin{definition}
Let $\mathbb{R}_+:=(0,\infty)$.  A
nonnegative function $\omega$ on $\mathbb{R}_+$
is said to be \emph{almost decreasing} on $\mathbb{R}_+$
if there exists a constant $C\in[1,\infty)$ such that,
for any $t, \tau\in(0,\infty)$ satisfying $t \geq \tau$,
$$
\omega(t)\leq C\omega(\tau).
$$
\end{definition}

\begin{definition}
Let $\mathbb{R}_+:=(0,\infty)$. The \emph{function class}
$M\left(\mathbb{R}_{+}\right)$ is defined to
be the set of all the positive functions
$\omega$ on $\mathbb{R}_{+}$ such that, for
any $0<\delta<N<\infty$,
$$
0<\inf _{t \in(\delta, N)} \omega(t) \leq
\sup _{t \in(\delta, N)} \omega(t)<\infty
$$
and there exist four constants $\alpha_{0}$,
$\beta_{0}$, $\alpha_{\infty}$,
$\beta_{\infty} \in \mathbb{R}$ such that
\begin{itemize}
\item[\rm (i)] for any $t \in(0,1]$, $\omega(t)
t^{-\alpha_{0}}$ is almost increasing and
$\omega(t) t^{-\beta_{0}}$ is almost
decreasing;

\item[\rm (ii)] for any $t \in[1, \infty)$, $\omega(t)
t^{-\alpha_{\infty}}$ is almost increasing
and $\omega(t) t^{-\beta_{\infty}}$ is
almost decreasing.
\end{itemize}
\end{definition}

Now, we recall the concept of local generalized Herz spaces
introduced in \cite[Definition 2.1]{rs20}.

\begin{definition}
Let $p,r \in(0, \infty)$ and $\omega \in
M\left(\mathbb{R}_{+}\right)$.
The \emph{local generalized Herz space}
$\dot{\mathcal{K}}_{\omega, \mathbf{0}}^{p,
r}(\mathbb{R}^{n})$ is defined to
be the set of all the measurable functions
$f$ on $\mathbb{R}^{n}$ such that
$$
\|f\|_{\dot{\mathcal{K}}_{\omega, \mathbf{0}}^{p,
r}\left(\mathbb{R}^{n}\right)}:=\left\{\sum_{k \in
\mathbb{Z}}\left[\omega\left(2^{k}\right)\right]^{r}\left\|f
\mathbf{1}_{B\left(\mathbf{0}, 2^{k}\right) \setminus
B\left(\mathbf{0},
2^{k-1}\right)}\right\|_{L^{p}\left(\mathbb{R}^{n}\right)}
^{r}\right\}^{\frac{1}{r}}<\fz.
$$
\end{definition}

We also recall the following concept of Matuszewska--Orlicz
indices; see, for instance, \cite{rs20} and
\cite[Definition 1.1.4]{LYH2022}

\begin{definition}
Let $\omega$ be a positive
function on $\mathbb{R}_{+}$. Then the
\emph{Matuszewska--Orlicz indices} $m_{0}(\omega)$,
$M_{0}(\omega)$, $m_{\infty}(\omega)$, and
$M_{\infty}(\omega)$ of $\omega$ are
defined, respectively, by setting, for any
$h \in(0, \infty)$,
$$m_{0}(\omega):=\sup _{t \in(0,1)}
\frac{\ln (\varlimsup\limits_{h \to
0^{+}} \frac{\omega(h
t)}{\omega(h)})}{\ln t},\ M_{0}(\omega):=\inf _{t \in(0,1)}
\frac{\ln (\varliminf\limits_{h\to0^{+}}
\frac{\omega(h t)}{\omega(h)})}{\ln
t},$$
$$m_{\infty}(\omega):=\sup _{t \in(1,
\infty)} \frac{\ln (\varliminf
\limits_{h \to\infty}
\frac{\omega(h t)}{\omega(h)})}{\ln
t},$$
and
$$
M_{\infty}(\omega):=\inf _{t \in(1, \infty)}
\frac{\ln (\varlimsup\limits_{h
\to \infty} \frac{\omega(h
t)}{\omega(h)})}{\ln t}.
$$
\end{definition}

\begin{remark}
From \cite[Theorem 1.2.20]{LYH2022},
we infer that $\dot{\mathcal{K}}_{\omega, \mathbf{0}}^{p,r}(\mathbb{R}^{n})$ is a ball quasi-Banach function space with $p,r\in(0,\infty)$ and $\omega\in M(\mathbb{R}_+)$ satisfying
$m_0(\omega)\in(-\frac{n}{p},\infty)$.
However, it is worth pointing out that
$\dot{\mathcal{K}}_{\omega, \mathbf{0}}^{p,r}(\mathbb{R}^{n})$ with
$p,r \in(0, \infty)$ and $\omega \in
M\left(\mathbb{R}_{+}\right)$ may not be a quasi-Banach function space
(see, for instance, \cite[Remark 4.13]{cjy22}).
\end{remark}

On the local
generalized Herz--Hardy space, we have the following Riesz
transform characterization.

\begin{theorem}\label{thm-re-mh}
Let $m\in\nn$, $p\in(\frac{n-1}{n+m-1},\infty)$, $r\in(\frac{n-1}{n+m-1},\infty]$,
$\omega\in M(\mathbb{R}_+)$ satisfy
$m_0(\omega)\in(-\frac{n}{p},\infty)$
and $m_\infty(\omega)\in(-\frac{n}{p},\infty)$, and $X:=\dot{\mathcal{K}}_{\omega, \mathbf{0}}^{p,r}(\mathbb{R}^{n})$.
Assume that
$$
\frac{n}{\max
\left\{M_{0}(\omega),
M_{\infty}(\omega)\right\}+n /
p}\in \left(\frac{n-1}{n+m-1},\infty\right).
$$
Then $H_X(\rn)=H_{X,{\rm Riesz}}^m(\rn)$
with equivalent (quasi-)norms.
\end{theorem}

\begin{proof}
By \cite[Theorems 1.2.20 and 1.4.1]{LYH2022},
we conclude that
$\dot{\mathcal{K}}_{\omega,\mathbf{0}}^{p,
r}(\mathbb{R}^{n})$ is a ball quasi-Banach space with an absolutely continuous quasi-norm.
To prove the required conclusion,
it suffices to show that
$\dot{\mathcal{K}}_{\omega,
\mathbf{0}}^{p,r}(\rn)$ satisfies
all the assumptions of both Theorems \ref{thm-re} and \ref{thm-h-re}. Indeed,
$\dot{\mathcal{K}}_{\omega,
\mathbf{0}}^{p,r}(\rn)$ satisfies Assumption
\ref{a} with
$$
s\in\left(\frac{n-1}{n+m-1},\min \left\{p, q, \frac{n}{\max
\left\{M_{0}(\omega),
M_{\infty}(\omega)\right\}+n /
p}\right\}\right]
$$
and
$$
\theta\in\left[\frac{n-1}{n+m-1}, s\right);
$$
see \cite[Lemma 4.3.9]{LYH2022}.
Moreover, from \cite[Lemma 1.8.5]{LYH2022}, we deduce
that Assumption \ref{a2} with $X:=\dot{\mathcal{K}}_{\omega,
\mathbf{0}}^{p,r}(\rn)$
holds true for any given
$$
q \in\left(\max \left\{1,p, \frac{n}{\min \left\{m_{0}(\omega),
m_{\infty}(\omega)\right\}+n / p}\right\}, \infty\right).
$$
Then, using Theorems \ref{thm-re}(iii) and \ref{thm-h-re}(iii)
with $X:=\dot{\mathcal{K}}_{\omega, \mathbf{0}}^{p,r}(\mathbb{R}^{n})$,
we obtain the desired conclusion, which
completes the proof of Theorem \ref{thm-re-mh}.
\end{proof}

\subsection{Mixed-Norm Herz
Space}\label{mhs}

In this section, we apply our main results to the mixed-norm Herz
spaces (see \cite{zyz2022}).

\begin{definition}\label{mhz}
Let $\vec{p}:=(p_{1},\ldots,p_{n}),\vec{q}:=(q_{1},\ldots,q_{n})
\in(0,\infty]^{n}$ and $\vec{\alpha}:=
(\alpha_{1},\ldots,
\alpha_{n})\in\rn$.
The \emph{mixed-norm Herz space}
$\dot{E}^{\vec{\alpha},\vec{p}}_{\vec{q}}(\rn)$ is
defined to be the
set of all the functions
$f\in \mathcal{M}(\rn)$ such
that
\begin{align*}
\|f\|_{\dot{E}^{\vec{\alpha},\vec{p}}_{\vec{q}}
(\rn)}:&=\left\{\sum_{k_{n} \in
\zz}2^{k_{n}
p_{n}\alpha_{n}}
\left[\int_{R_{k_{n}}}\cdots\left\{\sum_{k_{1}
\in \zz}
2^{k_{1}p_{1}\alpha_{1}}\right.\right.\right.\\
&\quad\times\left.\left.\left.\left[\int_{R_{k_{1}}}|f(x_{1},
\ldots,x_{n})|
^{q_{1}}\,dx_{1} \right]^{\f{p_{1}}{q_{1}}}
\right\}^{\f{q_{2}}
{p_{1}}}\cdots
\,dx_{n}  \right]^{\f{p_{n}}{q_{n}}}\right\}
^{\f{1}{p_n}}
<\infty
\end{align*}
with the usual modifications made when $p_{i}
=\infty$ or
$q_{j}=\infty$ for some $i,j\in\{1,\ldots,n \}$.
\end{definition}

\begin{remark}
Let $\vec{p},\vec{q}
\in(0,\infty]^{n}$ and $\vec{\alpha}\in\rn$.
By \cite[Propositions 2.8 and 2.22]{zyz2022}, we conclude that
$\dot{E}^{\vec{\alpha},\vec{p}}_{\vec{q}}(\rn)$
is a ball quasi-Banach space.
However, from \cite[Remark 2.4]{zyz2022}, we deduce that, when $\vec{p}=\vec{q}$ and
$\vec\alpha=\mathbf{0}$, the mixed-norm Herz
space $\dot{E}^{\vec{\alpha},\vec{p}}_{\vec{q}}(\rn)$
coincides with the mixed-norm Lebesgue
space $L^{\vec{p}}(\rn)$ defined in Definition \ref{mixed}.
Using Remark \ref{mix-r}, we conclude that $L^{\vec{p}}(\rn)$
may not be a quasi-Banach function space, and hence
$\dot{E}^{\vec{\alpha},\vec{p}}_{\vec{q}}(\rn)$
may not be a quasi-Banach function space.
\end{remark}

On the mixed-norm Herz--Hardy space, we have the following Riesz
transform characterization.

\begin{theorem}\label{thm-re-mhz}
Let $m\in\nn$, $\vec{p}:=(p_{1},\ldots,p_{n}),\vec{q}:=(q_{1},\ldots,q_{n})
\in(0,\infty)^{n}$, $\vec{\alpha}:=
(\alpha_{1},\ldots,
\alpha_{n})\in\rn$ with $\alpha_i\in(-\frac{1}{q_i},\infty)$ for any
$i\in\{1,\ldots,n\}$, and
$X:=\dot{E}^{\vec{\alpha},\vec{p}}_{\vec{q}}(\rn)$.
Assume that
$p_-,q_-\in(\frac{n-1}{n+m-1},\infty)$, where $p_-$ and
$q_-$ are the same as in \eqref{p-p+}, and, for any $i\in\{1,\ldots,n\}$,
$$
\lf(\alpha_{i}+\frac{1}{q_{i}}\r)^{-1}
\in\left(\frac{n-1}{n+m-1},\infty\right).
$$
Then $H_X(\rn)=H_{X,{\rm Riesz}}^m(\rn)$
with equivalent (quasi-)norms.
\end{theorem}

\begin{proof}
By \cite[Propositions 2.8 and 2.22]{zyz2022}, we conclude that
$\dot{E}^{\vec{\alpha},\vec{p}}_{\vec{q}}(\rn)$
is a ball quasi-Banach space with an absolutely continuous quasi-norm.
By \cite[Lemma 5.3(i)]{zyz2022},
we conclude that
$\dot{E}^{\vec{\alpha},\vec{p}}_{\vec{q}}(\rn)$ satisfies Assumption
\ref{a} with
$$s\in\left(\frac{n-1}{n+m-1},\min\lf\{1,p_-,
q_-,\lf(\alpha_{1}+\frac{1}{q_{1}}\r)^{-1},
\ldots,\lf(\alpha_{n}+\frac{1}{q_{n}}\r)^{-1} \r \}\right)$$
and
$$ \theta\in \left[\frac{n-1}{n+m-1},s\right).$$
Let $p_+$ and $q_+$ be the same as in \eqref{p-p+} and
$$
q\in\lf(\max\lf\{1, p_+,
q_+,\lf(\alpha_{1}
+\frac{1}{q_{1}}\r)^{-1}
,\ldots,\lf(\alpha_{n}+\frac
{1}{q_{n}}\r)^{-1}\r\},\infty \r).
$$
Then, by \cite[Lemma 5.3(ii)]{zyz2022}
and its proof,
we conclude that Assumption \ref{a2}
holds true with $X:=\dot{E}^{\vec{\alpha},\vec{p}}_{\vec{q}}(\rn)$.
Then, using Theorems \ref{thm-re}(iii) and \ref{thm-h-re}(iii)
with $X:=\dot{E}^{\vec{\alpha},\vec{p}}_{\vec{q}}(\rn)$,
we obtain the desired conclusion,
which completes the proof of
Theorem \ref{thm-re-mhz}.
\end{proof}

\subsection{Morrey Space}\label{ms}
In this section, we apply our main results to Morrey spaces.
Let us recall the concept of Morrey spaces.

\begin{definition}
Let $0<p\le r\le\fz$.
The \emph{Morrey space} ${\mathcal M}^r_p(\rn)$
is defined to be the set of all the
$f\in L^p_{{\rm loc}}(\rn)$ such that
\begin{equation*}
\|f\|_{{\mathcal M}^r_p(\rn)}:=\sup_{B\subset\rn}
|B|^{\frac1r-\frac1p}\left[\int_{B}|f(y)|^p\,dy\right]^{\frac1p}<\fz,
\end{equation*}
where the supremum is taken over all balls $B\subset\rn$.
\end{definition}

On the Morrey--Hardy space, we have the following Riesz
transform characterization.

\begin{theorem}\label{thm-re-mhre}
Let $m\in\nn$, $0<p\le r\le\fz$, and
$X:={\mathcal M}^r_p(\rn)$.
If
$p\in(\frac{n-1}{n+m-1},\infty)$, then the following two statements hold true.
\begin{enumerate}
\item[{\rm(i)}] There exists a positive constant $C\in[0,\infty)$
such that, for any $f\in L^2(\rn)$,
$$
 \|f\|_{H_{X,{\rm Riesz}}^m(\rn)}\leq C \|f\|_{H_X(\rn)}.
$$
\item[{\rm(ii)}] There exists a positive constant $C\in[0,\infty)$
such that, for any $f\in \cs'(\rn)$,
$$
\|f\|_{H_X(\rn)}\leq C  \|f\|_{H_{X,{\rm Riesz}}^m(\rn)}.
$$
\end{enumerate}
\end{theorem}

\begin{proof}
By \cite[Remarks 2.4(e) and 2.7(e)]{wyy20}, we conclude that
$X$
is a ball quasi-Banach space and satisfies Assumption
\ref{a} with
$$s\in\left(\frac{n-1}{n+m-1},\min\lf\{1,p \r \}\right)\
\text{and}
\
\theta\in \left[\frac{n-1}{n+m-1},s\right).
$$
Let $$
q\in\lf(\max\lf\{1, p\r\},\infty \r).
$$
Then, from \cite[Remarks 2.7(e)]{wyy20},
we deduce that Assumption \ref{a2}
holds true with $X:={\mathcal M}^r_p(\rn)$.
Then, using (i) and (ii) of Theorem \ref{thm-re} and (i) and (ii)
of Theorem \ref{thm-h-re} with $X:={\mathcal M}^r_p(\rn)$,
we obtain the desired conclusion,
which completes the proof of
Theorem \ref{thm-re-mhre}.
\end{proof}

\bigskip

\noindent Fan Wang, Dachun Yang (Corresponding author) and Wen Yuan

\medskip

\noindent Laboratory of Mathematics and Complex Systems (Ministry of Education of China),
School of Mathematical Sciences, Beijing Normal University,
Beijing 100875, People's Republic of China

\smallskip

\noindent{\it E-mails:} \texttt{fanwang@mail.bnu.edu.cn} (F. Wang)

\noindent\phantom{{\it E-mails:} }\texttt{dcyang@bnu.edu.cn} (D. Yang)

\noindent\phantom{{\it E-mails:} }\texttt{wenyuan@bnu.edu.cn} (W. Yuan)

\end{document}